\documentclass[11pt,letterpaper]{amsart}

\usepackage{amsfonts, amsmath, amssymb}
\usepackage{bbm}

\PassOptionsToPackage{hyphens}{url}\usepackage[backref,colorlinks,citecolor=blue,bookmarks=true]{hyperref}

\usepackage[hyphenbreaks]{breakurl}

\usepackage{tikz-cd}

\numberwithin{equation}{section}

\newtheorem{theorem}{Theorem}[section]
\newtheorem{lemma}[theorem]{Lemma}
\newtheorem{proposition}[theorem]{Proposition}
\newtheorem{corollary}[theorem]{Corollary}
\newtheorem{claim}[theorem]{Claim}
\newtheorem{remark}[theorem]{Remark}
\newtheorem{definition}[theorem]{Definition}

\newtheorem{conjecture}{Conjecture}

\newtheorem{main}{Theorem}
\newtheorem{main-coro}{Corollary}

\newtheorem{fact}[theorem]{Fact}

\newcommand{\V}[1]{\mathbf{{#1}}}
\newcommand{\B}[1]{\mathbf{{#1}}}
\newcommand{\T}{\mathbf{T}}
\DeclareMathOperator{\codim}{codim}
\DeclareMathOperator{\ch}{char}
\DeclareMathOperator{\rank}{rank}
\DeclareMathOperator{\Frob}{Frob}
\DeclareMathOperator{\Sing}{Sing}
\DeclareMathOperator{\I}{I}
\DeclareMathOperator{\D}{D}
\DeclareMathOperator{\J}{J}

\DeclareMathOperator{\AR}{AR}
\DeclareMathOperator{\GR}{GR}
\DeclareMathOperator{\PR}{PR}

\newcommand{\ceil}[1]{\left\lceil #1 \right\rceil}

\renewcommand{\d}{\mathrm{d}}
\newcommand{\tp}{t}

\renewcommand{\phi}{\varphi}
\renewcommand{\rho}{\varrho}

\newcommand{\FF}{\mathbb{F}}

\newcommand{\NN}{\mathbb{N}}

\newcommand{\KK}{\mathbb{K}}
\newcommand{\Comp}{\mathbb{C}}
\newcommand{\norm}[1]{\lVert #1 \rVert}

\newcommand{\eps}{\epsilon}
\newcommand{\sub}{\subseteq}

\newcommand{\f}{\Phi}

\DeclareMathOperator{\F}{\mathsf{F}}
\DeclareMathOperator{\Si}{\mathsf{S}}
\DeclareMathOperator{\Op}{\mathsf{O}}

\DeclareMathOperator{\Z}{Z}

\DeclareMathOperator{\Tr}{Tr}

\newcommand{\poly}{{\sf poly}}

\DeclareMathAccent{\wtilde}{\mathord}{largesymbols}{"65}

\DeclareMathOperator*{\Exp}{\mathbb{E}}
\DeclareMathOperator{\Ex}{\mathbb{E}}

\DeclareMathOperator{\image}{im}

\DeclareMathOperator{\adj}{adj}

\newcommand{\TS}[1]{\mathsf{#1}}
\renewcommand{\sl}[1]{\mathsf{#1}}

\DeclareMathOperator{\domain}{Dom}
\newcommand{\C}[2]{\mathsf{C}^{#1}(#2)}
\renewcommand{\O}{O}
\newcommand{\CC}[3]{\mathsf{C}^{#1}_{#3}(#2)}

\newcommand{\R}[1]{\mathfrak{#1}}

\newcommand{\bigtimes}{\prod}

\DeclareMathOperator{\IP}{IP}

\title{Partition and Analytic Rank are Equivalent over Large Fields}

\author{Alex Cohen}
\address{Massachusetts Institute of Technology, Cambridge, MA 02139, USA}
\email{alexcoh@mit.edu}
\author{Guy Moshkovitz}
\address{Department of Mathematics, City University of New York (Baruch College), New York, NY 10010, USA}
\email{guymoshkov@gmail.com}
\thanks{
	Part of this research was conducted during the 2020 NYC Discrete Math REU, supported by NSF awards DMS-1802059, DMS-1851420, and DMS-1953141}

\begin{document}

\maketitle

\begin{abstract}
	We prove that the partition rank and the analytic rank of tensors are
	equal up to a constant, over finite fields of any characteristic and any 
	large enough cardinality depending on the analytic rank.
	Moreover, we show that a plausible improvement
	of our field cardinality requirement would imply that the ranks are equal up to $1+o(1)$ in the exponent over every finite field.
	At the core of the proof is a technique for lifting decompositions of multilinear polynomials in an open subset of an algebraic variety,
	and a technique for finding a large subvariety 
	that retains all rational points such that at least one of these points satisfies a finite-field analogue of genericity with respect to it.	
	Proving the equivalence between these two ranks,
	ideally over fixed finite fields,
	is a central question in additive combinatorics, 
	and was reiterated by multiple authors.
	As a corollary we prove, allowing the field to depend on the value of the norm, the Polynomial Gowers Inverse Conjecture in the $d$~vs.~$d-1$ case.
\end{abstract}

\date{}

\maketitle



%

\section{Introduction}

The interplay between the structure and randomness of polynomials is a recurrent theme in combinatorics, analysis, computer science, and other fields.
A basic question in this line of research asks: 
if a polynomial is biased---in the sense that its output distribution deviates significantly from uniform---must it be the case that it is algebraically structured, in the sense that it is a function of a small number of lower-degree polynomials?
Quantifying the trade-off between ``biased'' and ``structured'' for polynomials has been the topic of many works, and has applications to central questions in  higher-order Fourier analysis, additive combinatorics, effective algebraic geometry, number theory, coding theory, and more
(see, e.g.,~\cite{BhowmickLo15,GowersWo11,GreenTao09,KazhdanZi20B} and the references within).


Let us state the question formally.
As is often done, we will only consider polynomials that are multilinear forms; it is known that the original question for degree-$k$ polynomials reduces to the case of $k$-linear forms, by symmetrization, with little loss in the parameters and over any field $\FF$ of characteristic $\ch(\FF)>k$ (see~\cite{GowersWo11,Janzer19} or Section~\ref{subsec:Gowers-proof} below).
%
%
%
A $k$-linear form, or a \emph{$k$-tensor}, over a field $\FF$ is a function $T \colon V_1 \times\cdots\times V_k \to \FF$, with $V_i$ finite-dimensional vector spaces over $\FF$, that is separately linear in each of its $k$ arguments; 
equivalently, $T$ is a degree-$k$ homogeneous polynomial of the form
$T(\B{x}_1,\ldots,\B{x}_k) = \sum_{I=(i_1,\ldots,i_k)} T_I x_{1,i_1}\cdots x_{k,i_k}$
with $T_I \in \FF$.
One may also identify a $k$-tensor $T$ with the $k$-dimensional array $(T_I)_I \in \FF^{n_1\times\cdots\times n_k}$, or with an element of the tensor space $V_1 \otimes\cdots\otimes V_k$.
Structure and randomness for tensors are defined using the following two notions:
\begin{itemize}
\item The \emph{partition rank} $\PR(T)$ of $T$ over a field $\FF$ is the smallest $r$ such that
$T$ is a sum of $r$ tensors of the form $pq$, 
where $p$ and $q$ are lower-order tensors over $\FF$ in disjoint sets of the $k$ arguments ($p \colon \bigtimes_{i \in S} V_i \to \FF$ and $q\colon \bigtimes_{i \in T} V_i \to \FF$ for a partition $S \cup T=[k]$ with $S,T\neq\emptyset$).
%
%
%
%
%
%
%
%
%
\item The \emph{analytic rank} $\AR(T)$ of $T$ over a finite field $\FF$ is\footnote{As is standard, the expectation notation stands for the normalized sum, $\Ex_{\B{x}\in S} f(\B{x}) = \sum_{\B{x}\in S} f(\B{x})/|S|$.} $$\AR(T) =-\log_{|\FF|} \Exp_{\B{x}\in V_1\times\cdots\times V_k} \chi(T(\B{x}))$$
for $\chi$ a nontrivial additive character (the definition is independent of the choice of $\chi$).
\end{itemize}
Analytic rank was introduced by Gowers and Wolf~\cite{GowersWo11} in the context of higher-order Fourier analysis, 
and partition rank of tensors was defined by Naslund~\cite{Naslund20} in the context of the cap set problem.
The easy direction is $\AR(T) \le \PR(T)$ (e.g., \cite{KazhdanZi18,Lovett19}).
The structure-vs-randomness question is the ``inverse'' problem, asking to bound $\PR(T)$ from above in terms of $\AR(T)$, 
or in other words, to find a small decomposition of any tensor exhibiting bias.



The seminal paper of Green and Tao~\cite{GreenTao09}, 
and subsequent refinements by Kaufman and Lovett~\cite{KaufmanLo08} and by Bhowmick and Lovett~\cite{BhowmickLo15}, 
prove the qualitative bound $\PR(T) \le f_k(\AR(T))$ with $f_k(x)$ of Ackermann-type growth.
Much better bounds were obtained for specific small values of $k$ by Haramaty and Shpilka~\cite{HaramatySh10}: $\PR(T) \le \AR(T)^4$ for $k=3$, and $\PR(T) \le 2^{O(\AR(T))}$ for $k=4$; Lampert~\cite{Lampert19} improved the bound for $k=4$ to a polynomial bound.
Back to $k=3$, 
a recent result of the authors~\cite{CohenMo21} gives $\PR(T) \le 3(1+o_{|\FF|}(1))\AR(T)$ if $\FF\neq \FF_2$; Adiprasito, Kazhdan, and Ziegler~\cite{AdiprasitoKaZi21} independently obtained a similar result using related ideas.
However, the problem of obtaining a reasonable bound for $k$-tensors remained elusive, until the groundbreaking works of Mili\'{c}evi\'{c}~\cite{Milicevic19}, and of Janzer~\cite{Janzer19} for small fields, 
who obtained polynomial bounds, roughly
$\PR(T) \le (\AR(T)+1)^D$ with $D = 2^{2^{O(k^2)}}$.
%
Multiple authors (e.g., \cite{AdiprasitoKaZi21,HaramatySh10,KazhdanZi20B,Lovett19}) have asked whether---or conjectured that---
$\PR(T) = \Theta_k(\AR(T))$ for every $k$-tensor $T$;
in other words, that structure and randomness of low-degree polynomials are two sides of the same coin.
 
\begin{conjecture}[Partition~vs.~Analytic Rank Conjecture]\label{conj:equiv}
	For every $k$-tensor $T$ over any finite field, 
	$\PR(T) \le O_k(\AR(T))$.
\end{conjecture}

Conjecture~\ref{conj:equiv}
is considered far stronger than known results
and, if true, could give sharp bounds in a myriad of applications
(see, e.g.,~\cite{AnanyanHo20-B,BhowmickLo15,GowersWo11,GreenTao09} and the references within). 
Over algebraically closed fields, an analogue of this conjecture is known~\cite{Schmidt84,Schmidt85}
(see the discussion at the end of this section).
Results over algebraically closed fields can typically be translated (via model-theoretic arguments; see, e.g., Remark~4 in~\cite{Tao09} and the references within) to results over finite fields of large enough characteristic---although large enough here means 
growing with the number of variables.
Therefore, a fundamental and essential challenge  
is to prove the bound in Conjecture~\ref{conj:equiv}
over every finite field whose characteristic---or, better yet, just cardinality---is bounded from below independently of the number of variables.

%



Let us briefly mention that for many 
algebraic questions, 
proving a ``uniform'' bound---by which we mean a bound independent of the number of variables---can be notoriously difficult.
For example, Stillman's conjecture posits that any $r$ polynomials of degrees at most $k$ generate an ideal of projective dimension\footnote{Intuitively, projective dimension measures the ``depth'' of the iterated algebraic relations between the polynomials.} bounded in terms of only $r$ and $k$; or put differently, that there is a uniform proof of Hilbert’s Syzygy Theorem.
Stillman's conjecture was famously proved by Ananyan and Hochster~\cite{AnanyanHo20}, relying on a structure-vs-randomness inequality 
(see Theorem~A(a) in~\cite{AnanyanHo20}, or~\cite{ErmanSaSn19}). 
Proving a sharp version of Stillman's conjecture remains open~\cite{AnanyanHo20-B}; indeed, like many other applications, proving a sharp bound seems to require Conjecture~\ref{conj:equiv} as a prerequisite.


In this paper we prove Conjecture~\ref{conj:equiv} over every finite field whose cardinality is large enough in terms of the analytic rank (in any characteristic).
\begin{main}\label{theo:main}
	For every $k \ge 2$ and $r \ge 0$ there is $F=F(r,k)$ such that the following holds for every finite field $\FF$ of cardinality $|\FF| \ge F$.
	For every $k$-tensor $T$ over $\FF$ with $\AR(T) \le r$, 
	$$\PR(T) \le (2^{k-1}-1)\AR(T) + 1.$$
\end{main}

Our bound on the field size is double-exponential: 
$F(r,k) = 2^{2^{O(r+k)}}$. 
Although the dependence of $F$ on $r$ means that Conjecture~\ref{conj:equiv} is not yet proved for constant-size fields (if it is indeed true for such fields), 
we note that the independence of $F$ from the number of variables $N$ implies that 
if $\PR(T)$ and $k$ do not go to infinity with $N$ (which is the interesting case in most applications) 
then $\AR(T)$ does not either---as $\AR(T) \le \PR(T)$---and so
the upper bound on $\PR(T)$ in Theorem~\ref{theo:main} holds for all but boundedly many finite fields.
%
%
As for very small fields (e.g., $\FF_2$), 
we show that improving the dependence on $r$ to sub-exponential
would imply a nearly sharp bound 
without any assumption on the finite field (cardinality or characteristic). 

\begin{main-coro}\label{coro:main}
	If Theorem~\ref{theo:main} holds with $F(r,k) \le \exp(O_k(r^{o(1)}))$ then we have
	$\PR(T) \le O_k(\AR(T)^{1+o(1)})$ over every finite field.
\end{main-coro}

This should be compared with the previously best known bound, which we recall is  $\PR(T) \le O_k(\AR(T)^D)$ with $D = 2^{2^{O(k^2)}}$~\cite{Janzer19,Milicevic19}. 
In fact, even milder improvements to $F(r,k)$ could lead to an improvement of the state-of-the-art bounds over e.g.\ $\FF_2$, suggesting a potential research path in the case of very small fields.





\subsection{Polynomial bound for the $d$~vs.~$d-1$ Gowers Inverse Conjecture}

An important conjecture in 
additive combinatorics is the \emph{Polynomial Gowers Inverse Conjecture}, which posits that the inverse theorem for the Gowers uniformity norms over finite fields holds with a polynomial bound (see, e.g.,~\cite{GreenTao08,Lovett12,Wolf15}, where it was conjectured for $U^3$ norms). 
The best known bound for the $U^d$ inverse theorem over finite fields, by Gowers and Mili\'{c}evi\'{c}~\cite{GowersMi20}, is roughly an exponential tower of height $d!$.\footnote{Over the integers, Manners~\cite{Manners18} obtained a double-exponential bound.}
As a corollary of Theorem~\ref{theo:main}, we make progress towards this conjecture in the \emph{$d$~vs.~$d-1$} case, originally raised by Bogdanov and Viola~\cite{BogdanovVi10},\footnote{A closely-related conjecture is sometimes called the ``Inverse conjecture for polynomials''~\cite{Tao11-char} (or $\text{GIP}$ in~\cite{TaoZi12}).} which already suffices for some applications of the full Gowers Inverse Theorem.

For the rest of this subsection $\FF$ is a finite field, and
$\chi \colon \FF\to\Comp$ is a nontrivial additive character (e.g., the standard character $\chi(y) = (e^{2\pi i/|\FF|})^y$
for $|\FF|$ prime).
For a function $f\colon\FF^n\to\Comp$ we denote by $\Delta^*_{\B{v}}f(\B{x})=f(\B{x}+\B{v})\overline{f(\B{x})}$ the \emph{multiplicative derivative} of $f$ along $\V{v} \in \FF^n$.
The \emph{Gowers norm} $\norm{f}_{U^d}$, and the \emph{weak Gowers norm} $\norm{f}_{u^d}$, are given by
\begin{align}\label{eq:Gowers-norms}
	\begin{split}
\norm{f}_{U^d} &= |\Ex_{\B{v}_1,\ldots,\B{v}_d,\B{x}\in\FF^n} \Delta^*_{\B{v}_1}\cdots\Delta^*_{\B{v}_d} f(\B{x})|^{1/2^d}
\quad\text{ and }\quad\\
\norm{f}_{u^d} &= \max_{\deg(q)<d} |\Exp_{\B{x}\in\FF^n} f(\B{x})\chi(-q(\B{x}))|
\end{split}
\end{align}
where the maximum is over the polynomials $q \colon \FF^n \to \FF$ of degree at most $d-1$.
Thus, $\norm{f}_{U^d}$ measures the average bias of order-$d$ multiplicative derivatives of $f$, 
while $\norm{f}_{u^d}$ measures the largest correlation of $f$ with a \emph{polynomial phase function} of degree smaller than $d$.
These are indeed norms for $d\ge 2$, and seminorms otherwise.
It is known that $\norm{f}_{u^d} \le \norm{f}_{U^d} \le 1$ for every function $f \colon \FF^n \to \Comp$ with $|f(\B{x})|\le 1$ (see~\cite{GreenTao09}).
%
The Polynomial Gowers Inverse Conjecture (PGI) over finite fields, in the high characteristic case, is as follows.
\begin{conjecture}[PGI, high characteristic case]\label{conj:PGI}
	Let $d \ge 1$.
	For every function $f \colon \FF^n \to \Comp$ with $|f(\B{x})|\le 1$, over any finite field $\FF$ with $\ch(\FF) > d$, $\norm{f}_{u^d} \ge \poly_{d}(\norm{f}_{U^d})$.
\end{conjecture}

We next state the $d$~vs.~$d-1$ Gowers Inverse Conjecture of Bogdanov and Viola (Conjecture~21 in~\cite{BogdanovVi10})---now a theorem, over fields of characteristic at least $d$, following the works of Green and Tao~\cite{GreenTao09} and Bhowmick and Lovett~\cite{BhowmickLo15}.
Unlike the general case (proved in works of Bergelson, Tao, and Ziegler~\cite{BergelsonTaoZi10,TaoZi10,TaoZi12}), it involves the $d$-th Gowers norms only for polynomial phase functions whose degree is $d$.

\vskip 5pt

\paragraph{\textbf{$d$~vs.~$d-1$ Gowers Inverse Theorem over $\FF$}:}
	For every $\delta > 0$ there exists $\eps > 0$ such that for every polynomial $p\colon \FF^n\to\FF$ of degree $d$, if 
	$f(\B{x}):=\chi(p(\B{x}))$ 
	satisfies $\norm{f}_{U^d} \ge \delta$ then $\norm{f}_{u^d} \ge \eps$.\\

Conjecture~\ref{conj:PGI}, or its special cases, is of central importance.
In the $d$~vs.~$d-1$ case, quantitative bounds follow from bounds for the partition~vs.~analytic rank question (see, e.g.,~\cite{Janzer19}).
Theorem~\ref{theo:main} can be used to prove the $d$~vs.~$d-1$ inverse theorem (in the high characteristic case) with a polynomial bound.

\begin{main-coro}\label{coro:Gowers}
	Let $d \ge 1$.
	For every $\delta \ge 0$ there exists $F=F(\delta,d)\in\NN$ such that 
	for every finite field $\FF$ with $\ch(\FF)>d$ and $|\FF|\ge F$,
	and for every polynomial $p\colon \FF^n\to\FF$ of degree $d$, 
	if 
	$f(\B{x}):=\chi(p(\B{x}))$ 
	satisfies $\norm{f}_{U^d} \ge \delta$ then 
	$\norm{f}_{u^d} \ge \delta^{8^d}$.
%
%
%
%
%
\end{main-coro}

\subsection{Proof overview}\label{subsec:overview}

We begin our proof, as mathematicians often do, by
extending our notion of numbers, 
replacing the finite field $\FF$ in Conjecture~\ref{conj:equiv} with its algebraic closure $\overline{\FF}$,
which enables the use of tools from algebraic geometry. 
The main drawback of this approach, as one might assume, is that finding a small partition rank decomposition over $\overline{\FF}$---that is, one where the summands are tensors over $\overline{\FF}$---need not yield a small partition rank decomposition over $\FF$, which is what Conjecture~\ref{conj:equiv} calls for.

The analogue of analytic rank over an algebraically closed field $\overline{\FF}$ is the \emph{geometric rank}. 
%
To define it, we view a $k$-tensor $T \colon \prod_{i=1}^k\overline{\FF}^{n_i} \to \overline{\FF}$ as a $(k-1)$-linear map $\sl{T} \colon \prod_{i=1}^{k-1}\overline{\FF}^{n_i} \to \overline{\FF}^{n_k}$; this is done by considering the ``slices'' of the $k$-dimensional array $T \in \overline{\FF}^{n_1\times\cdots \times n_k}$, along the $k$-th axis, as $(k-1)$-linear forms (for $k=2$ this is the familiar correspondence between matrices and linear maps).
%
The geometric rank $\GR(T)$ is the codimension of the kernel variety of $\sl{T}$,
$$\ker\sl{T} = \Big\{ \B{x} \in \prod_{i=1}^{k-1}\overline{\FF}^{n_i} \,\Big\vert\, \sl{T}(\B{x})=\B{0} \Big\}.$$
%
Geometric rank was studied in~\cite{KoppartyMoZu20}, 
where in particular it was shown to be independent of the choice of axis to slice $T$ along; for $k=2$ this is column rank equals row rank.
(A similar geometric notion was also used in the context of number theory~\cite{Schmidt84,Schmidt85}.)
That $\GR(T)$ is the analogue of $\AR(T)$ over $\overline{\FF}$ is suggested by the identity $\AR(T) = (\sum_{i=1}^{k-1} n_i) - \log_{|\FF|}|\V{\ker\sl{T}}(\FF)|$, where $\ker\sl{T}(\FF)$ is the set of $\FF$-rational points in $\ker\sl{T}$.
Thus, 
the analogue of Conjecture~\ref{conj:equiv} over $\overline{\FF}$ asks to bound an algebraic notion of complexity---partition rank---by a geometric notion of complexity---the codimension of the kernel variety.

In order to circumvent the drawback of moving to $\overline{\FF}$, an important idea in our proof is that we not only find a partition rank decomposition for $T$ but also for a whole family of related tensors of various orders: the iterated total derivatives $\D\sl{T},\D^2\sl{T},\ldots,\D^{k-1}\sl{T}$ of $\sl{T}$ evaluated at all points in an open subset of a subvariety 
of $\ker\sl{T}$.
The total derivative $\D\sl{T}$ is given by the Jacobian matrix, meaning we replace each entry of $\sl{T}$ by the row vector of its partial derivatives.
(For example, if $\sl{T}(\B{x},\B{y})=(x_1y_1,\ldots,x_ny_n)$ then $\D\sl{T}$ is the $n \times 2n$ matrix whose $i$-th row is $y_i\B{e}_i + x_i\B{e}_{n+i}$.)
Note that $\D\sl{T}$ is a \emph{polynomial matrix}, meaning a matrix whose entries are polynomials; more generally, 
$\D^i\sl{T}$ is a \emph{polynomial $(i+1)$-tensor}.
For each derivative $\D^i \sl{T}$ we construct a rational map that takes each point $\B{x}$ in the open subset to some partition rank decomposition of the $(i+1)$-tensor $\D^i \sl{T}(\B{x})$.
This rational map into the space of partition rank decompositions, or a \emph{rational decomposition} for short, 
is obtained by looking at the tangent spaces to the subvariety at the various points of the open subset,
decomposing the total derivative into a term in the direction of the tangent vectors and a complementary term.

At the core of our result is a lifting argument, showing that if a polynomial tensor $f$ has a rational decomposition on an open subset of an irreducible variety $\V{V}$, then the total derivative $\D f$ has a rational decomposition, on an open subset of $\V{V}$, of complexity roughly bounded by the codimension of $\V{V}$ (see Theorem~\ref{theo:induction-step} for the precise statement).
Since $\D f$ is again a polynomial tensor, this enables an iterative lifting process.
One can apply this process starting with the polynomial matrix $\D\sl{T}$ and the kernel variety $\ker\sl{T}$. For the base case of this process, however, we must first obtain a rational decomposition of $\D\sl{T}$ by other means. This turns out to require a rational map computing a rank factorization of the polynomial matrix $\D\sl{T}$, which we explicitly construct.
By the end of the iterative process, we have a rational decomposition, on an open subset of $\ker\sl{T}$, for the polynomial $k$-tensor $\D^{k-1}\sl{T}$.
Since the total derivative retains all the partial derivatives throughout the iterations, and since each entry of $\sl{T}$ is a polynomial of degree $k-1$, we can reconstruct $T$ from $\D^{k-1}\sl{T}$ by evaluating the latter at any point. 
Now, one can obtain a partition rank decomposition over $\overline{\FF}$ by simply evaluating the resulting rational decomposition at any point from the open subset of $\ker\sl{T}$ on which it is defined.
This gives the following bound.

\begin{main-coro}\label{theo:main2}
	For every $k$-tensor $T$ over every algebraically closed field,
	$$\PR(T) \le (2^{k-1}-1)\GR(T).$$
\end{main-coro}




Crucially, our proof method is flexible enough to yield a partition rank decomposition over the ground finite field $\FF$. 
Instead of relying on the usual definition of the dimension of a variety $\V{V}$, we consider a variant $\dim_\FF \V{V}$ which, roughly speaking, is the largest dimension of a subvariety $\V{Z}$, with $\V{V}(\FF) \sub \V{Z} \sub \V{V}$, that contains an $\FF$-rational point satisfying certain genericity conditions.\footnote{To be more precise, this definition also depends on the set of polynomials cutting out the variety.}
Clearly, $\dim_\FF \V{V} \le \dim \V{V}$.
By applying our iterative lifting argument on an appropriate irreducible component of the subvariety witnessing $\dim_{\FF} \ker\sl{T}$, and evaluating the resulting rational decomposition at the guaranteed point, 
we obtain an upper bound on $\PR(T)$ in terms of $\codim_{\FF} \ker\sl{T}$ (see Theorem~\ref{theo:PR-bound} for the precise statement). Unlike the case where $\FF$ is algebraically closed (recall Corollary~\ref{theo:main2}), our proof does not yield an upper bound of $\PR(T)$ in terms of $\codim\ker\sl{T}=\GR(T)$;
nevertheless, this is enough for our purposes.
In the last part of the proof of Theorem~\ref{theo:main} (see Section~\ref{sec:final}) we show that for $\V{V}=\ker\sl{T}$, there is a subvariety $\V{Z}$ as above of controlled complexity. 
This is shown by iteratively applying certain operations on the top-dimensional irreducible components, 
and requires defining a variant of the usual notion of degree of a variety (see Definition~\ref{def:deg-growth}).
Since $\V{Z} \supseteq \ker\sl{T}(\FF)$, 
this allows us to deduce, over any (mildly) large finite field, a lower bound on $\AR(T)$ in terms of $\codim\V{Z} \ge \codim_{\FF} \ker\sl{T}$, thus 
yielding the chain of inequalities:
$$\PR(T) \le O_k(\codim_{\FF} \ker\sl{T}) \le O_k(\AR(T)) .$$
It is worth remarking that our arguments in Section~\ref{sec:final} do not guarantee the existence of an $\FF$-rational point on the variety $\ker\sl{T}$, but rather rely on the fact that they already exist (and not only $\B{0} \in \ker\sl{T}$, as otherwise $\AR(T)$ is maximal so there is nothing to prove),
and instead show that at least one of them is ``generic'' in some appropriate sense.
As is evident from the overview above, our proof is self-contained 
and does not rely on results from additive combinatorics, nor any regularity lemma for polynomials or notions of quasi-randomness.
This makes our proof quite different from some arguments 
previously applied in this line of research.



\subsection*{Comparison with the work of Schmidt.\,} 
Over the complex numbers, Corollary~\ref{theo:main2} (with a constant of about $k!$ instead of our $2^{k-1}-1$) can be obtained from a proof found in a 1985 paper of Schmidt~\cite{Schmidt85}.
In~\cite{Schmidt84}, Schmidt also mentions (without full proof) that his result extends to general algebraically closed fields, as in Corollary~\ref{theo:main2}. 
%
%
In~\cite{Schmidt85}, Schmidt considers the number of integer solutions of undetermined systems of polynomial equations, and gives several conditions under which such a system has roughly the ``expected'' number of solutions in axis-parallel integer boxes of varying size.
Interestingly, in the course of his proof, Schmidt defines analogues of partition rank and geometric rank, and gives an essentially algebraic proof that implies, in our terminology, the bound $\PR(T) \le O_k(\GR(T))$ over the complex numbers,
which he uses to obtain his main, number-theoretic results.
Schmidt's arguments do not give a proof of such a bound over finite fields, nor does it seem that minor modifications to Schmidt's proof would be sufficient to do so.
While both Schmidt's proof and ours utilize derivatives at a rational point, 
they differ in several key ways, which we expand upon more technically at the end of Section~\ref{sec:decomposition}. 
Our methods seem 
necessary
for obtaining results over finite fields.

\subsection*{Subsequent work.} 
Shortly after our paper appeared, Kazhdan and Polishchuk~\cite{KazhdanPo21} adapted Schmidt's proof~\cite{Schmidt85} to arbitrary fields, but their bound is weaker for fields that are not algebraically closed, and is vacuous for finite fields. 
In our terminology, they show that for a $k$-tensor $T$ over a field $\FF$, 
$\PR(T) \le C_k\cdot\codim \overline{\ker\sl{T}(\FF)}$,
where $\overline{\ker\sl{T}(\FF)}$ is the Zariski closure of the set of $\FF$-rational points in $\ker\sl{T}$
(with $C_k$ roughly $k!$, similarly to~\cite{Schmidt85}).
If $\FF$ is a finite field, the set $\ker\sl{T}(\FF)$ is finite and thus a variety of dimension $0$, yielding a trivial upper bound on $\PR(T)$ (for $T \colon \prod_{i=1}^k\FF^{n_i}\to\FF$ we get $\PR(T) \le C_k(n_1+\cdots+n_{k-1})$).

\subsection{Paper organization}
In Section~\ref{sec:constructing} we specify the space in which our decompositions live, and relate it to partition rank decompositions.
In Section~\ref{sec:total-derivative} we
 show that a total derivative of a tensor of sufficiently high order contains enough information to reconstruct the original tensor.
In Section~\ref{sec:LA} we give a rational map for a rank factorization of matrices (used in the base case of our inductive proof) and a rational map parameterizing the kernel of matrices of a given rank (used in the induction step).
In Section~\ref{sec:tangents} we obtain a rational map parameterizing the tangent space to a variety at non-singular points.
In Section~\ref{sec:decomposition} we give our inductive proof, which yields a partition rank decomposition assuming our irreducible variety is defined over $\FF$ and has an $\FF$-rational point satisfying certain non-singularity conditions. We also obtain Corollary~\ref{theo:main2} here.
In Section~\ref{sec:final} we complete the proof of Theorem~\ref{theo:main} by constructing, over large enough fields, a variety satisfying the properties above, which, moreover, can be used to bound the analytic rank of our tensor.
Moreover, we give here a proof of Corollary~\ref{coro:main}.
Finally, in Subsection~\ref{subsec:Gowers-proof} we obtain the polynomial bound for the $d$~vs.~$d-1$ inverse theorem for the Gowers norms stated in Corollary~\ref{coro:Gowers}.

\subsection{Notation}

An \emph{$\FF$-polynomial} is a polynomial with coefficients in $\FF$ (i.e., an element of $\FF[\B{x}]=\FF[x_1,\ldots,x_n]$).
A \emph{rational map} $f \colon \FF^n \dashrightarrow \FF^m$ is an $m$-tuple of rational functions over $\FF$ (i.e., quotients of $\FF$-polynomials).
A rational map determines a function whose domain $\domain(f)$ is the set of inputs for which all $m$ rational functions are defined. 
We use the convention that $f \colon \FF^n \dashrightarrow \FF^m$ may be evaluated at points $\B{x} \in \overline{\FF}^n$ over the closure, in which case $f(\B{x}) \in \overline{\FF}^m$ if $\B{x} \in\domain(f)$ ($\sub \overline{\FF}^n$).
%
%
For a positive integer $n \in \NN$ we denote $[n]=\{1,2,\ldots,n\}$.
All vectors spaces are finite dimensional.

We will need some basic algebro-geometric terminology.
We say that a variety $\V{V} \sub \overline{\FF}^n$ is \emph{defined} over $\FF$ if it can be cut out by some $\FF$-polynomials $f_1,\ldots,f_m$, meaning
$$\V{V} = 
\{ \B{x} \in \overline{\FF}^n \mid  f_1(\B{x})=\cdots=f_m(\B{x})=0\} .$$
The \emph{(vanishing) ideal} of $\V{V}$ is $\I(\V{V})=\{f \in \overline{\FF}[x_1,\ldots,x_n] \mid \forall \B{x} \in \V{V} \colon f(\B{x})=0\}$.
Any variety $\V{V}$ can be uniquely written as the union of \emph{irreducible} varieties, where a variety is said to be irreducible if it cannot be written as the union of strictly contained varieties.
The \emph{dimension} of a variety $\V{V}$, denoted $\dim\V{V}$, is the maximal length $d$ of a chain of irreducible varieties 
$\emptyset \neq \V{V}_1 \subsetneq\cdots\subsetneq \V{V}_d \subsetneq \V{V}$.
The \emph{codimension} of $\V{V} \sub \overline{\FF}^n$ is simply $\codim\V{V}=n-\dim\V{V}$.
A point $\B{x} \in \V{V}$ is said to be \emph{$\FF$-rational} if $\B{x} \in \FF^n$; we denote by $\V{V}(\FF)$ the set of $\FF$-rational points in $\V{V}$.

Crucial to our proof is the simultaneous use of different aspects of a $k$-tensor.
Henceforth, we fix bases for our vector spaces, so that we will freely pass between a vector space $V$ and its dual $V \to \FF$, 
between matrices in $V_1 \otimes V_2$ and linear maps $V_1 \to V_2$, and more generally, between $k$-tensors in $V_1 \otimes\cdots\otimes V_k$ and $k$-linear forms $V_1 \times\cdots\times V_{k} \to \FF$  or maps $V_1 \times\cdots\times V_{k-1} \to V_k$.
If we identify $V_k$ with $\FF^m$ and $V_1 \times\cdots\times V_{k-1}$ with $\FF^n$ $(m=\dim V_k,\, n=\dim V_1+\cdots+\dim V_{k-1})$,
this means that any $k$-linear form $T \colon V_1 \times\cdots\times V_k \to \FF$ corresponds to a polynomial map $\sl{T} \colon \FF^n \to \FF^m$, in fact a $(k-1)$-linear map, which we refer to as the \emph{slicing} of $T$ (along $V_k$), and denote with a slanted typeface $\sl{T}$.



\section{Constructing Space}\label{sec:constructing}

Let $\TS{V} = \bigotimes_{i=1}^k V_i$ be a space of $k$-tensors $(k \ge 2)$ over a field $\FF$.
The \emph{$r$-constructing space} $\C{r}{\TS{V}}$ over $\TS{V}$ is the vector space
$$\C{r}{\TS{V}} = \CC{r}{\TS{V}}{1} \times \CC{r}{\TS{V}}{2} =
\bigg(\bigoplus_{\emptyset \neq S \sub [k-1]} \TS{V}^{\otimes S} \bigg)^r
\times \bigg(\bigoplus_{\emptyset \neq S \sub [k-1]} \TS{V}^{\otimes \overline{S}} \bigg)^r .$$
where $\TS{V}^{\otimes X} := \bigotimes_{i \in X} V_i$ for $X \sub [k]$,\footnote{
	We henceforth identify tensor products taken in different orders, e.g., $V_1 \otimes V_2 \otimes V_3$ with $V_2 \otimes V_1 \otimes V_3$.}
and where $\overline{S}:=[k]\setminus S$.
Note that the tensor products in $\CC{r}{\TS{V}}{1}$ correspond to nonempty subsets of $[k]$ that do not contain $k$, whereas the tensor products in $\CC{r}{\TS{V}}{2}$ correspond to subsets of $[k]$ that do contain $k$.
Moreover, note that the tensor products in $\CC{r}{\TS{V}}{1}$, together with their counterparts in $\CC{r}{\TS{V}}{2}$,
correspond to unordered, non-trivial partitions $\{S,T\}$ of $[k]$ (i.e., $S \cup T = [k]$ with $S,T \neq \emptyset$), of which there are $2^{k-1}-1$.
For example, the order-$2$ and order-$3$ $r$-constructing spaces are, respectively,
$$\C{r}{V_1 \otimes V_2} = V_1^r \times V_2^r \quad\text{ and}$$
$$\C{r}{V_1 \otimes V_2 \otimes V_3} =
\Big(V_1 \oplus V_2 \oplus (V_1 \otimes V_2)\Big)^r
\times \Big((V_2 \otimes V_3) \oplus (V_1 \otimes V_3) \oplus V_3\Big)^r .
$$
Let $\IP \colon (\bigoplus_{i\in[n]} U_i) \times (\bigoplus_{i\in[n]} W_i) \to \sum_{i\in[n]} U_i \otimes W_i$, for vector spaces $U_i$ and $W_i$, be the function
\begin{equation}\label{eq:IP}
	\IP(\B{u}_1,\ldots,\B{u}_{n},\B{w}_1,\ldots,\B{w}_{n})
	= \B{u}_1 \otimes \B{w}_1 + \cdots + \B{u}_n \otimes \B{w}_n .
\end{equation}
%
Observe that applying $\IP$ on the $r$-constructing space $\C{r}{\TS{V}}$ gives a map 
into the tensor space $\TS{V}$:
$$\IP \colon (\B{u}_{j,S},\B{w}_{j,\overline{S}})_{\substack{j \in [r],\,\emptyset \neq S \sub [k-1]}}
\mapsto \sum_{\substack{j\in[r]\\\emptyset \neq S \sub [k-1]}} \B{u}_{j,S} \otimes \B{w}_{j,\overline{S}} \in \TS{V} $$
with $\B{u}_{j,S} \in \TS{V}^{\otimes S}$ and $\B{w}_{j,\overline{S}} \in  \TS{V}^{\otimes \overline{S}}$.
Since the number of sets $\emptyset \neq S \sub [k-1]$ is $2^{k-1}-1$ 
we have, immediately from definition, that if a $k$-tensor lies in the image of $\IP$ on the $r$-constructing space then its partition rank is at most $(2^{k-1}-1)r$.
\begin{fact}\label{fact:PR-image}
	Let $T \in \TS{V} = \bigotimes_{i=1}^k V_i$ 
	with $V_i$ vector spaces over $\FF$.
	If $T \in \IP(\C{r}{\TS{V}})$ then $\PR(T) \le (2^{k-1}-1)r$.
\end{fact}


\section{Total Derivative}\label{sec:total-derivative}

The \emph{total derivative} of a rational map $f=(f_1,\ldots,f_m) \colon \FF^n \dashrightarrow \FF^m$
is the rational map $\D f \colon \FF^n \dashrightarrow \FF^m \otimes \FF^n$ given by
$$\D f(\B{x}) = (\partial_j f_i(\B{x}))_{i\in[m], j\in[n]} .$$
Note that $\D f$ is a matrix whose every entry is a rational function.
The directional derivative of $f$ along $\B{v} \in \FF^n$ is thus given by
$$(\D f(\B{x}))\B{v} 
= (\partial_1 f_i(\B{x})v_1+\cdots+\partial_n f_i(\B{x})v_n)_{i\in[m]} \in \FF^m.$$
%
Observe that we may iterate the total derivative operator. For any integer $a \in \NN$, the order-$a$ total derivative of a rational map
$f=(f_1,\ldots,f_m) \colon \FF^n \dashrightarrow \FF^m$
is the rational map $\D^a f \colon  \FF^n \dashrightarrow \FF^m \otimes (\FF^n)^{\otimes a}$ given by
$$\D^a f(\B{x}) = (\partial_{j_1,\ldots,j_a} f_i(\B{x}))_{i\in[m], j_1,\ldots,j_a\in[n]} \;.$$
In particular, $\D^a f$ is an $(a+1)$-dimensional array, or $(a+1)$-tensor, whose every entry is a rational function.
%
%
%
%
We will need the following properties of the total derivative:
\begin{itemize}
	\item Sum rule: $\D(f+g) = \D f + \D g$, where $f,g \colon \FF^n \dashrightarrow \FF^m$.
	\item Product rule: $\D(fg) = g\D f + f\D g$, where $f,g \colon \FF^n \dashrightarrow \FF$.
	\item Quotient rule: $\D(f/g) = (1/g^2)(g\D f - f\D g)$, where $f,g \colon \FF^n \dashrightarrow \FF$.
	\item Chain rule: $\D(f \circ g)(\B{x}) = (\D f)(g(\B{x})) \D g(\B{x})$, 
 	where $g \colon \FF^n \dashrightarrow \FF^m$ and $f \colon \FF^m \dashrightarrow \FF^t$.
	\item The rational maps $f$ and $\D^a f$ have the same domain.
\end{itemize}
Note that if $f$ is a (single) polynomial of degree $d$ then $\D^a f$ is a polynomial $a$-tensor whose every entry is of degree $d-a$.
In particular, every entry of $\D^{d} f$ is a constant polynomial.

Recall that the slicing of a $k$-tensor $T \in \bigotimes_{i=1}^k V_i$ is a polynomial map $\sl{T} \colon \FF^{n} \to \FF^{m}$ (with $m = \dim V_k$ and
$n=\sum_{j=1}^{k-1} \dim V_j$).
Thus, for any $a \in \NN$ and $\B{x_0} \in \FF^n$, the order-$a$ total derivative
$\D^a \sl{T}(\B{x_0}) \in \FF^m \otimes (\FF^n)^{\otimes a}$ at $\B{x_0}$ is an $(a+1)$-tensor.

\begin{proposition}\label{claim:PR-restriction}
	Let $T$ be a $k$-tensor over $\FF$,
	and let $\sl{T} \colon \FF^n \to \FF^m$ be the slicing of $T$.
	For every $\B{x_0} \in \FF^n$, the $k$-tensor $T'=\D^{k-1} \sl{T}(\B{x_0})$ satisfies
	$\PR(T) \le \PR(T')$.
\end{proposition}
\begin{proof}
	We claim that the $(k-1)$-linear maps $\sl{T} \colon \FF^n \to \FF^m$ and $\sl{T'} \colon (\FF^n)^{k-1} \to \FF^m$, obtained as slicings of the $k$-tensors $T$ and $T' \in \FF^m \otimes (\FF^n)^{k-1}$, respectively, 
	satisfy the identity 
	\begin{equation}\label{eq:PR-restriction}
	\sl{T}(\B{x}_1,\ldots,\B{x}_{k-1}) = \sl{T'}((\B{x}_1,\B{0},\ldots,\B{0}),(\B{0},\B{x}_2,\ldots,\B{0}),\ldots,(\B{0},\B{0},\ldots,\B{x}_{k-1})).
	\end{equation}
	This would imply that $\PR(T) \le \PR(T')$,
	since the identity $\sl{T}(\B{x}) = \sl{T'}(\B{x'})$ in~(\ref{eq:PR-restriction}) 
	means that for every $\B{x} \in \FF^n$ and $\B{z} \in \FF^m$
	we have
	$T(\B{x},\B{z}) = \sl{T}(\B{x}) \cdot \B{z} 
	= \sl{T'}(\B{x'}) \cdot \B{z} 
	= T'(\B{x'},\B{z})$.
%
%
	
	Put $d=k-1$, and write $\sl{T} = (\sl{T}_1,\ldots,\sl{T}_m)$ with 
	$\sl{T}_\ell(\B{y}_1,\ldots,\B{y}_d) 
	= \sum_{j_1,\ldots,j_d} c_{j_1,\ldots,j_d}^{(\ell)} y_{1,j_1}\cdots y_{d,j_d}$ where $c_{j_1,\ldots,j_d}^{(\ell)} \in \FF$.
	By definition,
	$$\sl{T'}(\B{x}^{(1)},\ldots,\B{x}^{(d)})
	= \Big( \sum_{(i_1,j_1),\ldots,(i_d,j_d)} \frac{\partial^d \sl{T}_\ell(\B{x_0})}{\partial y_{i_1,j_1}\cdots\partial y_{i_d,j_d}} \, x^{(1)}_{i_1,j_1}\cdots x^{(d)}_{i_{d},j_{d}}\Big)_{\ell\in[m]}.$$
	We have
	$$\frac{\partial^d \sl{T}_\ell}{\partial y_{i_1,j_1}\cdots\partial y_{i_d,j_d}} 
	= \begin{cases}
	c_{j_{\tau^{-1}(1)},\ldots,j_{\tau^{-1}(d)}}^{(\ell)}	& i_1,\ldots,i_d \in[d] \text{ distinct}\\
	0 	& \text{otherwise} 
	\end{cases}
	$$
	where $\tau$ is the permutation $(i_1\, i_2\,\ldots\,i_d)$ (i.e., $x \mapsto i_x$).
	Setting $\B{x}^{(t)} = (\B{0},\ldots,\B{0},\B{x}_t,\B{0},\ldots,\B{0})$ for $t=1,\ldots,d$ 
	(i.e., $x^{(t)}_{i,j} = x_{i,j}\cdot \mathbbm{1}(t=i)$), 
	we get
	\begin{align*}
	\sl{T'}((\B{x}_1,\B{0},\ldots,\B{0}),\ldots,(\B{0},\ldots,\B{0},\B{x}_{d}))
	&= \Big( \sum_{j_1,\ldots,j_d} \frac{\partial^d \sl{T}_\ell(\B{x_0})}{\partial y_{1,j_1}\cdots\partial y_{d,j_d}} \, x_{1,j_1}\cdots x_{d,j_{d}}\Big)_{\ell\in[m]}\\
	&= \Big( \sum_{j_1,\ldots,j_d} c^{(\ell)}_{j_1,\ldots,j_d} \, x_{1,j_1}\cdots x_{d,j_{d}} \Big)_{\ell\in[m]}\\
	&= \sl{T}(\B{x}_1,\ldots,\B{x}_d).
	\end{align*}
	This gives~(\ref{eq:PR-restriction}) and thus completes the proof.
\end{proof}

Recall the function $\IP(\vec{\B{u}},\vec{\B{w}})\colon (\bigoplus_{i\in[n]} U_i) \times (\bigoplus_{i\in[n]} W_i) \to \sum_{i\in[n]} U_i \otimes W_i$ from~(\ref{eq:IP}),
and note that we may view it as a polynomial map by fixing bases for the vector spaces $U_i$ and $W_i$.
The next claim gives the directional derivative of $\IP$.



\begin{proposition}\label{claim:DIP}
	For every $(\vec{\B{x}},\vec{\B{y}}) \in (\bigoplus_{i\in[n]} U_i) \times (\bigoplus_{i\in[n]} W_i)$, 
	$$(\D \IP(\vec{\B{u}},\vec{\B{w}}))(\vec{\B{x}},\vec{\B{y}}) 
	= \IP(\vec{\B{u}},\vec{\B{y}}) + \IP(\vec{\B{w}}, \vec{\B{x}}).$$
	%
	%
\end{proposition}
\begin{proof}
	By linearity it suffices to prove the claim for $n=1$. 
	Viewing $\IP \colon U \times W \to U \otimes W$ as a polynomial map $\IP \colon \FF^{\dim U} \times \FF^{\dim W} \to \FF^{\dim U\times\dim W}$, we have
	$\IP(\B{u},\B{w}) = \B{u} \otimes \B{w} = (u_iw_j)_{i,j}$. 
	Thus, for every $(\B{x},\B{y}) \in \FF^{\dim U} \times \FF^{\dim W}$,
	$$(\D\IP(\B{u},\B{w}))(\B{x},\B{y}) = (w_jx_i + u_i y_j)_{i,j}
	= (w_jx_i)_{i,j} + (u_i y_j)_{i,j}
	= \IP(\B{w},\B{x}) + \IP(\B{u},\B{y}),$$
	as needed.
	%
	%
\end{proof}

\section{Rational Linear Algebra}\label{sec:LA}

For a matrix $A \in \FF^{m \times n}$, we denote by $A_{x\times y}=(A_{i,j})_{i\in[x],j\in[y]}$ the submatrix of $A$ consisting of the first $x$ rows and the first $y$ columns.
We denote by $I_r$ the $r \times r$ identity matrix.

A \emph{rank factorization} of a matrix $A \in \FF^{m \times n}$ of rank $r$ is a pair of matrices $(X,Y) \in \FF^{m \times r} \times \FF^{r \times n}$ such that $A=XY$.
We next show that rank factorization can be done via a rational map,
mapping (an open subset of) matrices of rank exactly $r$ to a rank-$r$ factorization thereof.

\begin{lemma}\label{lemma:rank-factorization}
	Let $m,n,r\le \min\{m,n\} \in \NN$ and $\FF$ any field.
	There is a rational map $F \colon \FF^{m \times n} \dashrightarrow \FF^{m \times r} \times \FF^{r \times n}$ 
	such that $F(A)=(X,Y)$ is a rank factorization of $A$, for every $A \in \domain(F) = \{ A \mid \det A_{r \times r} \neq 0 \}$ with $\rank A = r$; 
	moreover, $X=A_{m\times r}$ and $Y_{r \times r} = I_r$.
\end{lemma}
%
%
\begin{proof}
	We first prove the following matrix identity for matrices $A \in \FF^{m \times n}$ 
	with $\rank(A)=r$ and $A_{r\times r}$ invertible;
	\begin{equation}\label{eq:matrix-identity}
	A = A_{m\times r} (A_{r\times r})^{-1}A_{r\times n} .
	\end{equation}
	Since the column rank of $A$ equals $\rank(A)=r$, the first $r$ columns of $A$ are not just linearly independent but a basis for the column space.
	Thus, there is a matrix $Y \in \FF^{r \times n}$ such that
	$$A_{m \times r}Y = A .$$
	Restricting this equality to the first $r$ rows, we have
	$A_{r \times r}Y = A_{r \times n}$. 
	Since $A_{r\times r}$ is invertible, 
	$$Y = (A_{r \times r})^{-1} A_{r \times n} .$$
	Combining the above gives the identity~(\ref{eq:matrix-identity}).
	
	Let $F \colon \FF^{m \times n} \dashrightarrow \FF^{m \times r} \times \FF^{r \times n}$ be the rational map 
	$$F(A) = (A_{m\times r},\, \det(A_{r \times r})^{-1}\adj(A_{r\times r})A_{r\times n}) ,$$
	where we recall that each entry of the adjugate matrix $\adj(A_{r\times r})$ is a cofactor and thus a polynomial in the entries of $A_{r\times r}$.
	We have 
	that $(A_{r\times r})^{-1} = \det(A_{r \times r})^{-1}\adj(A_{r\times r})$.
	Thus, $\domain(F) = \{ A \mid \det A_{r \times r} \neq 0 \}$ and, 
	by~(\ref{eq:matrix-identity}), $F(A)$ is a rank factorization for any $A \in \domain(F)$ with $\rank A = r$. 
	Finally, for the ``moreover'' part, note
	that $(A_{r\times r})^{-1}A_{r\times n}$ contains as its first $r$ columns the submatrix $(A_{r\times r})^{-1}A_{r\times r} = I_r$.
\end{proof}

Recall that a matrix $P \in \FF^{n \times n}$ is a projection matrix onto a vector subspace $U \sub \FF^n$ if $P^2=P$ and $\image P=U$.
Moreover, recall that if $P$ is a projection matrix then $I_n-P$ is a projection matrix onto $\ker P$. 
Next, we show that projecting onto the kernel of a matrix can be done by a rational map.

\begin{lemma}\label{lemma:rational-projection}
	Let $m,n,r \le \min\{m,n\} \in \NN$ and $\FF$ any field.
	There is a rational map $P \colon \FF^{m \times n} \dashrightarrow \FF^{n \times n}$ 
	such that $P(A)$ is a projection matrix onto $\ker A$, 
	for every $A \in \domain(P) = \{ A \mid \det A_{r \times r} \neq 0 \}$ with $\rank A = r$; 
	moreover, 
	in the projection matrix 
	$I_n-P(A)$ all but the first $r$ rows are zero.
\end{lemma}
\begin{proof}
	For a matrix $A \in \FF^{r \times n}$, denote by $\overline{A} \in \FF^{n \times n}$ the matrix obtained from $A$ by adding $n-r$ zero rows at the bottom.
	Let $F \colon \FF^{m \times n} \dashrightarrow \FF^{m \times r} \times \FF^{r \times n}$ be the rank-factorization rational map given by Lemma~\ref{lemma:rank-factorization}.
	We take $P \colon \FF^{m \times n} \dashrightarrow \FF^{n \times n}$ 
	to be the rational map $P(A)=I_n - \overline{Y}$ where $F(A)=(X,Y)$.
	Note that $\domain(P)=\domain(F)$.
	Let $A \in \domain(P)$ with $\rank A=r$. 
	By Lemma~\ref{lemma:rank-factorization},
	the matrix $\overline{Y}$ is of the form
	\begin{equation}\label{eq:proj}
		\overline{Y} = 
		\begin{pmatrix}
			I_r & Y' \\
			\B{0} & \B{0}
		\end{pmatrix}_{n \times n} 
	\end{equation}
	for some matrix $Y' \in \FF^{r\times (n-r)}$. 
	One can easily check that $\overline{Y}^2=\overline{Y}$ and that $\ker \overline{Y} = \ker Y$. 
	We claim that $\ker Y = \ker A$; indeed, 
	$\ker Y \sub \ker A$ since $A=XY$, 
	and $\dim\ker Y = \dim\ker A$ since $\rank Y = r = \rank A$.
	We deduce that 
	$P(A) = I_n - \overline{Y}$ is a projection matrix onto $\ker\overline{Y} = \ker A$.
	This completes the proof.
\end{proof}

\section{Tangent spaces}\label{sec:tangents}

The \emph{tangent space} of a variety $\V{V} \sub \overline{\FF}^n$ at $\B{x} \in \V{V}$ is the vector space
$$\T_\B{x}\V{V} = \Big\{ \B{v} \in \overline{\FF}^n \,\Big\vert\, \forall g \in \I(\V{V}) \colon (\D g(\B{x}))\B{v} = 0 \Big\}.$$ 
Equivalently, if $\{g_1,\ldots,g_m\} \sub \overline{\FF}[x_1,\ldots,x_n]$ is a generating set for the ideal $\I(\V{V})$ then the tangent space to $\V{V}$ at $\B{x} \in \V{V}$ is the kernel
$\T_\B{x} \V{V} = \ker \J(\B{x})$, where $\J(\B{x})$ is the Jacobian matrix
$$\begin{pmatrix}
(\partial_{x_1} g_1)(\B{x}) & \cdots & (\partial_{x_n}g_1)(\B{x}) \\
\vdots  & \ddots & \vdots  \\
(\partial_{x_1} g_m)(\B{x}) & \cdots & (\partial_{x_n} g_m)(\B{x})
\end{pmatrix}_{m\times n} .$$
A point $\B{x}$ on an irreducible variety $\V{V}$ is \emph{non-singular} if
$\dim \T_\B{x} \V{V} = \dim \V{V}$.
%
We have the following basic fact about tangent spaces (see, e.g., Theorem~2.3 in~\cite{Shafarevich}).
\begin{fact}\label{fact:dim-tangent}
	For every irreducible variety $\V{V}$ and $\B{x} \in \V{V}$,
	$\dim \T_\B{x} \V{V} \ge \dim \V{V}$.
\end{fact}

Recall that a field is \emph{perfect} if each element has a root of order $\ch(\FF)$ or if $\ch(\FF)=0$. 
In particular, finite fields, characteristic-zero fields, and algebraically closed fields are all perfect.
%
%
%
We will use the following known fact (see, e.g., Remark~21.2.7 in~\cite{FriedJa}).

\begin{fact}\label{lemma:F-generators}
	If an irreducible variety $\V{V}$ is defined over a perfect field $\FF$
	then $\I({\V{V}})$ has a (finite) generating set of $\FF$-polynomials.
\end{fact}
\begin{proof}[Proof sketch]
	This can be deduced from Lemma~10.2.3 in~\cite{FriedJa} (which goes back to~\cite{Lang}, page~74), which tells us that if an irreducible variety $\V{V}$ is cut out by $\FF$-polynomials then $\I(\V{V})$ can be generated by polynomials over a finite purely inseparable extension of $\FF$.
	If $\FF$ is perfect then its only purely inseparable extension is itself (see, e.g., Chapter 1.5.6 in~\cite{LutharPa04}).
	The finiteness of the generating set follows from Hilbert's basis theorem.
%
%
%
\end{proof}



%
%
%
%

We next show that over a perfect field, tangent spaces---at non-singular points in a nontrivial open subset---can be parameterized by a rational map.


\begin{lemma}\label{lemma:rational-tangent}
	Let $\V{V} \sub \overline{\FF}^n$ be an irreducible variety defined over a perfect field $\FF$,
	and $\B{x_0}$ be a non-singular 
	point on $\V{V}$.
	There is a rational map $\R{P} \colon \FF^{n} \dashrightarrow \FF^{n \times n}$ 
	that is defined at $\B{x_0}$ 
	such that $\R{P}(\B{x})$ is a projection matrix onto $\T_\B{x} \V{V}$ for every $\B{x} \in \V{V} \cap \domain(\R{P})$;
	moreover, in the projection matrix $I_n-\R{P}(\B{x})$ all but the first $\codim\V{V}$ rows are zero.
\end{lemma}
\begin{proof}
	As $\FF$ is perfect, by Fact~\ref{lemma:F-generators} the ideal $\I(\V{V})$ has a generating set of $\FF$-polynomials, $\{g_1,\ldots,g_m\} \sub \FF[\B{x}]$.
	Let $\J \colon \B{x} \mapsto (\partial_{j} g_i(\B{x}))_{i \in [m], j\in[n]}$ be the corresponding Jacobian matrix, viewed as a polynomial map $\J \colon \FF^n \to \FF^{m \times n}$.
	By definition, $\ker\J(\B{x})=\T_\B{x}\V{V}$ for every $\B{x} \in \V{V}$.
	We will take $\R{P}(\B{x})$ to be the projection map from Lemma~\ref{lemma:rational-projection} onto $\ker\J(\B{x})$, making sure it is defined at $\B{x_0}$.
	
	Put $r=\codim\V{V}$.
	Since $\B{x_0}$ is a non-singular point on $\V{V}$, we have $\rank(\J(\B{x_0})) = \codim\V{V} = r$. In particular, the submatrix $\J(\B{x_0})_{I \times J}$ is invertible for some $I \sub [m]$ and $J \sub [n]$ with $|I|=|J|=r$.
	Apply Lemma~\ref{lemma:rational-projection} with $m,n,r$, together with appropriate row and column permutations, to obtain a rational map $P \colon \FF^{m \times n} \dashrightarrow \FF^{n \times n}$ such that 
	$P(A)$ is a projection matrix onto $\ker A$, 
	for any matrix $A \in \domain(P) = \{ A \mid \det A_{I \times J} \neq 0 \}$ with $\rank A=r$.
	%
	Let $\R{P} \colon \FF^n \dashrightarrow \FF^{n \times n}$ be the rational map $\R{P} = P\circ \J$.
	By construction, $\B{x_0} \in \domain(\R{P})$, and $\R{P}(\B{x}) = P(\J(\B{x}))$ is a projection matrix onto $\ker \J(\B{x})$,
	for every $\B{x} \in \domain(\R{P})$ with $\rank \J(\B{x})=r$.
	By Fact~\ref{fact:dim-tangent}, $\rank \J(\B{x}) \le r$ for every $\B{x} \in \V{V}$.
	It follows that 
	$\R{P}(\B{x})$ is a projection matrix onto $\T_\B{x} \V{V}$ for every $\B{x} \in \V{V} \cap \domain(\R{P})$; 
	moreover, $I_n-\R{P}(\B{x})$ has zeros in all but the first $r$ rows, by Lemma~\ref{lemma:rational-projection}.
	This completes the proof.
\end{proof}

\section{Partition Rank Decomposition}\label{sec:decomposition}

Recall that a polynomial matrix is polynomial map into matrices (or a matrix whose every entry is a polynomial).
\begin{definition}[$M$-singularity]\label{def:M-sing}
	For a polynomial matrix $M \colon \FF^n \to \FF^{a \times b}$,
	we say that a point $\B{x'}$ on an irreducible variety $\V{X} \sub \overline{\FF}^n$ is \emph{$M$-singular} if $\rank M(\B{x'}) < \max_{\B{x} \in \V{X}} \rank M(\B{x})$.
\end{definition}

Note that $M$-singularity extends the usual notion of a singular point, which is obtained by taking $M$ to be the Jacobian of a generating set of $\I(\V{X})$.
In this section we prove the following bound on the partition rank---over any perfect field.
Recall that the slicing of a $k$-tensor $T$ is denoted by $\sl{T}$, 
%
and thus $\D\sl{T}$ is a polynomial matrix
(if $T \colon \FF^{n_1} \times\cdots\times \FF^{n_k} \to \FF$ 
then $\sl{T} \colon \FF^N \to \FF^{n_k}$ with $N=\sum_{i=1}^{k-1} n_i$ and $\D\sl{T} \colon \FF^N \to \FF^{n_k \times N}$).
\begin{theorem}\label{theo:PR-bound}
	Let $T$ be a $k$-tensor over a perfect field $\FF$.
	Let $\V{X} \sub \ker\sl{T}$ 
	be an irreducible variety defined over $\FF$ that has an $\FF$-rational point that is neither singular nor $\D\sl{T}$-singular on $\V{X}$.
	Then 
	$$\PR(T) \le (2^{k-1}-1)\codim\V{X}.$$
%
\end{theorem}

When the field $\FF$ in Theorem~\ref{theo:PR-bound} is algebraically closed, taking $\V{X}$ to be any top-dimensional irreducible component of $\ker\sl{T}$, together with a generic point on $\V{X}$, immediately gives the bound $\PR(T) \le (2^{k-1}-1)\GR(T)$, proving Corollary~\ref{theo:main2}.

The proof of Theorem~\ref{theo:PR-bound} starts by
obtaining a rational decomposition for the polynomial matrix $\D\sl{T}$ using the rank-factorization map from Lemma~\ref{lemma:rank-factorization}, then proceeds to iteratively construct a rational decomposition for each of $\D^2\sl{T},\D^3 \sl{T},\ldots,\D^{k-1} \sl{T}$ on an open subset of $\V{X}$---by parameterizing the tangent spaces of $\V{X}$ and taking derivatives along them and along their complementary subspaces---and finally evaluates the rational decomposition at the guaranteed 
$\FF$-rational point to obtain a partition rank decomposition over $\FF$ for $T$ itself.


The following theorem gives the inductive step in the proof of Theorem~\ref{theo:PR-bound}.
For rational maps $f$ and $g$ we write $f|_{\V{V}} = g|_{\V{V}}$ if $f(\B{x})=g(\B{x})$ for every $\B{x} \in \V{V}$ for which $f(\B{x})$ and $g(\B{x})$ are defined.


\begin{theorem}\label{theo:induction-step}
	Let $\V{V} \sub \overline{\FF}^n$ be an irreducible variety defined over a perfect field $\FF$, $\B{x_0}$ a non-singular point on $\V{V}$, 
	and $f \colon \FF^n \to \bigotimes_{i=1}^k V_i =: \TS{V}$ 
	a polynomial map with $k\ge 2$.
	%
	If there exists a rational map $H \colon \FF^n \dashrightarrow \C{r}{\TS{V}}$ defined at $\B{x_0}$ 
	such that $f|_\V{V} = \IP\circ H|_\V{V}$ 
	then there exists a rational map $H' \colon \FF^n \dashrightarrow \C{s}{\TS{V}\otimes \FF^n}$
	defined at $\B{x_0}$
	such that $(\D f)|_\V{V} = \IP\circ H'|_\V{V}$
	with $s = \max\{r,\,\codim \V{V}\}$.
\end{theorem}

Theorem~\ref{theo:induction-step} can be illustrated using the commutative diagrams:
$$
\begin{tikzcd}
\V{V} \arrow[r, dashrightarrow, "H"] \arrow[d, hook] & \C{r}{\TS{V}} \arrow[d, "\IP"] \\
\overline{\FF}^n \arrow[r, "f"] & \TS{V}
\end{tikzcd}
\quad\Longrightarrow\quad
\begin{tikzcd}
\V{V} \arrow[r, dashrightarrow, "H'"] \arrow[d, hook] & \C{s}{\TS{V}\otimes \FF^n} \arrow[d, "\IP"] \\
\overline{\FF}^n \arrow[r, "\D f"] & \TS{V\otimes \FF^n} 
\end{tikzcd}
$$
The proof of Theorem~\ref{theo:induction-step} proceeds by splitting the derivative $\D f$ into two parts: a main term which is mapped into the portion of the constructing space on $\TS{V}\otimes \FF^n$ where each product involves a tensor ``mixed'' between $\TS{V}$ and $\FF^n$,
and a remainder term---controlled by the codimension of $\V{V}$---which is mapped into the complementary portion of the constructing space where each product involves a tensor in $\TS{V}$ and a $1$-tensor in $\FF^n$.
We note that the curvature of the variety $\V{V}$ enters into the partition rank decomposition---specifically, the remainder term---through the derivative along complementary subspaces to the tangent spaces.

Before going into the proof, 
we remind the reader that if $g \colon \FF^n \dashrightarrow \TS{V}$ is a rational map into a space $\TS{V}$ of $k$-tensors, 
then 
the total derivative $\D g(\B{x}) \in \TS{V} \otimes \FF^n$ is a $(k+1)$-tensor, while the directional derivative of $g$ along any vector $\B{v}\in \FF^n$ is again a $k$-tensor, $(\D g(\B{x}))\B{v} \in \TS{V}$.
We refer the reader to Remark~\ref{remark:boring} below for an instructive example application of Theorem~\ref{theo:induction-step}.


\begin{proof}[Proof of Theorem~\ref{theo:induction-step}]
	For $\B{h} \in \C{r}{\TS{V}}$ we henceforth write $\B{h}=(\B{h}_1,\B{h}_2) \in \CC{r}{\TS{V}}{1} \times \CC{r}{\TS{V}}{2}$.
	Using the chain rule for the total derivative, we obtain the identity
	\begin{align}
		\begin{split}\label{eq:chain}
			\D(\IP\circ H)(\B{x})
			&= (\D \IP)(H(\B{x}))\D H(\B{x})\\
			&= (\D \IP)(H_1(\B{x}),H_2(\B{x}))(\D H_1(\B{x}), \D H_2(\B{x}))\\
			&= \IP(H_1(\B{x}),\D H_2(\B{x})) + \IP(H_2(\B{x}),\D H_1(\B{x})) ,
		\end{split}
	\end{align}
	where the last equality follows from Proposition~\ref{claim:DIP}. 
	
	Put $\O = \V{V} \cap \domain(H)$. 
	Since $\B{x_0} \in \O$, we have that $\O$ is a non-empty open subset of the irreducible variety $\V{V}$, and thus its Zariski closure is $\overline{\O}=\V{V}$.
	%
	Taking the total derivative at any $\B{x} \in \O$ along a tangent vector $\B{v} \in \T_\B{x}\V{V}$, 
	we claim that
	\begin{equation}\label{eq:D-tangent}
	(\D f(\B{x}))\B{v} = (\D\IP\circ H(\B{x}))\B{v}.
	\end{equation}
	Indeed, if $p/q$ ($p,q$ are $\FF$-polynomials) is any one of the entries of $f - \IP\circ H$ then, since $(f - \IP\circ H)(\O)=0$, we have $p(\O)=0$.
	Thus, $p(\V{V})=p(\overline{\O})=0$.
	Since $(\D p(\B{x}))\B{v}=0$ by the definition of $\T_\B{x}\V{V}$, and since
	$$\D\frac{p}{q}(\B{x}) 
	= \frac{1}{q(\B{x})^2}(q(\B{x})\D p(\B{x}) - p(\B{x})\D q(\B{x}))
	= \frac{1}{q(\B{x})}\D p(\B{x}),$$	
	we have $(\D(p/q)(\B{x}))\B{v}=0$, and so 
	$(\D(f - \IP\circ H)(\B{x}))\B{v}=0$, verifying~(\ref{eq:D-tangent}).

	Apply Lemma~\ref{lemma:rational-tangent}, with $\V{V}$ and the non-singular $\B{x_0}$, to obtain a rational map $P \colon \FF^{n} \dashrightarrow \FF^{n \times n}$ 
	that is defined at $\B{x_0}$, such that $P(\B{x})$ is a projection matrix onto $\T_\B{x} \V{V}$ for every $\B{x} \in \V{V} \cap \domain(P)$.
	Let $Q(\B{x}) = I_n - P(\B{x})$. 
	Put $\O' = \V{V} \cap \domain(H) \cap \domain(P)$, and note that $\B{x_0} \in \O'$.
	Let $A_i$ denote column $i$ of a matrix $A$.
	We deduce that for every $\B{x} \in \O'$ we have
	\begin{align}
	\begin{split}\label{eq:D-decomposition}
	\D f(\B{x})
	&= (\D f(\B{x}))(P(\B{x}) + Q(\B{x}))
	= (\D\IP \circ H(\B{x}))P(\B{x}) + (\D f(\B{x})) Q(\B{x})\\
	&= \IP(H_1(\B{x}),(\D H_2(\B{x}))P(\B{x})) 
	+ \IP(H_2(\B{x}),(\D H_1(\B{x}))P(\B{x})) \\
	&+ \sum_{i=1}^{\codim\V{V}} (\D f(\B{x}))Q(\B{x})_i \otimes Q^\tp(\B{x})_i
	\end{split}
	\end{align}
	where the second equality follows from~(\ref{eq:D-tangent}),
	and the last equality applies~(\ref{eq:chain}) and, moreover, uses the decomposition of the projection matrix $Q=Q(\B{x})$ as (recall Lemma~\ref{lemma:rational-tangent})
	$Q = Q^2 
	= \sum_{i=1}^{\codim\V{V}} Q_i \otimes (Q^\tp)_i$. 

	Put $r'=\codim\V{V}$. 
	We will now use the rational map $H$,
	which maps to the $k$-constructing space $\C{r}{\TS{V}}$, 
	in order to construct a rational map $H'$ to the $(k+1)$-constructing space $\C{r}{\TS{V} \otimes V_{k+1}}$, where $V_{k+1}$ is a vector space isomorphic to $\FF^n$.
	Let $H'(\B{x}) = (H'_1(\B{x}), H'_2(\B{x}))$ be given by
	$$H'_1(\B{x}) = H_1(\B{x}) \oplus H_2(\B{x}) \oplus ((\D f(\B{x}))Q(\B{x})_i)_{i=1}^{r'}\, , $$
	$$H'_2(\B{x}) = (\D H_2(\B{x}))P(\B{x}) \oplus (\D H_1(\B{x}))P(\B{x}) \oplus (Q^\tp(\B{x})_i)_{i=1}^{r'}.\footnote{The number of tensor products involved in $H'_1$ (and in $H'_2$) is twice the number in $H_1$ plus $1$, that is, $2(2^{k-1}-1)+1=2^k-1$.}$$
	Note that for every $\B{x} \in \O' \cap \FF^n$ we have
	$$H'_1(\B{x}) \in \Big(\bigoplus_{\emptyset \neq S \sub [k-1]} (\TS{V}^{\otimes S})^r \oplus \bigoplus_{\emptyset \neq S \sub [k-1]} (\TS{V}^{\otimes [k]\setminus S})^r \oplus (\TS{V}^{\otimes [k]})^{r'} \Big) \;,$$
	$$H'_2(\B{x}) \in \Big( \bigoplus_{\emptyset \neq S \sub [k-1]} (\TS{V}^{\otimes [k]\setminus S} \otimes V_{k+1})^r \oplus \bigoplus_{\emptyset \neq S \sub [k-1]} (\TS{V}^{\otimes S} \otimes V_{k+1})^r \oplus (V_{k+1})^{r'} \Big) .
	$$
	Put $\TS{U} = \bigotimes_{i=1}^{k+1} V_i = \TS{V} \otimes V_{k+1}$.
	Observe that $H'_1(\B{x})$ only involves tensor products that do not include $V_{k+1}$, whereas $H'_2(\B{x})$ only involves tensor products that do include $V_{k+1}$.
	Therefore
	$$
	H'(\B{x}) \in \bigg(\bigoplus_{\emptyset \neq S \sub [k]} \big(\TS{U}^{\otimes S}\big)^{\max\{r,r'\}} \bigg)
	\times \bigg(\bigoplus_{\emptyset \neq S \sub [k]} \big(\TS{U}^{\otimes [k+1]\setminus S}\big)^{\max\{r,r'\}} \bigg)
	= \CC{s}{\TS{U}}{1} \times \CC{s}{\TS{U}}{2} .
	$$
	Thus, $H' \colon \FF^n \dashrightarrow \C{s}{\TS{U}}$ is a rational map 
	with $\domain(H') = \domain(H) \cap \domain(P)$.
	We deduce that $H'$ is defined at $\B{x_0}$ and, by~(\ref{eq:D-decomposition}), satisfies for every $\B{x} \in \V{V} \cap \domain(H')$ that
	$$\D f(\B{x}) = \IP(H'(\B{x})) .$$
	This completes the proof.
\end{proof}

We note that Theorem~\ref{theo:induction-step} obtains a rational decomposition for the restriction of the derivative $(\D f)|_{\V{V}}$, rather than for the derivative of the restriction $\D(f|_{\V{V}})$,
and so iterating it does not lead to losing any information about $f$---even if, for example, $(\D^i f)|_{\V{V}}=0$ for some $i$.
Iterating Theorem~\ref{theo:induction-step} gives the following.
	
\begin{corollary}\label{coro:induction-step}
	Let $\V{V} \sub \overline{\FF}^n$ be an irreducible variety defined over a perfect field $\FF$, $\B{x_0}$ a non-singular point on $\V{V}$,
	and $M \colon \FF^n \dashrightarrow A \otimes B$ a polynomial matrix.
	If there exists a rational map $H \colon \FF^n \dashrightarrow \C{r}{A \otimes B}$ 
	defined at $\B{x_0}$ such that
	$M|_{\V{V}} = \IP\circ H|_{\V{V}}$ 
	then, for every $a \in \NN$, there exists a rational map $H' \colon \FF^n \dashrightarrow \C{s}{A \otimes B \otimes (\FF^n)^{\otimes a}}$
	defined at $\B{x_0}$ such that
	$(\D^a M)|_{\V{V}} = \IP\circ H'|_{\V{V}}$
	with $s = \max\{r,\,\codim \V{V}\}$.
\end{corollary}
\begin{proof}
	Let $H \colon \FF^n \dashrightarrow \C{r}{A \otimes B}$ be a rational map
	defined at $\B{x_0}$ such that $M|_{\V{V}} = \IP\circ H|_{\V{V}}$. 
	We proceed by induction on $a \ge 0$. 
	The induction base $a=0$ clearly holds 
	with $H'=H$, 
	since $(\D^0 M)|_{\V{V}} = M|_{\V{V}} = \IP\circ H'|_{\V{V}}$,
	hence we move to the induction step.
	By the induction hypothesis, there exists a rational map $H' \colon \FF^n \dashrightarrow \C{s}{A \otimes B \otimes (\FF^n)^{\otimes a}}$
	defined at $\B{x_0}$,
	with $s = \max\{r,\,\codim \V{V}\}$, 
	such that $(\D^a M)|_{\V{V}}  = \IP\circ H'|_{\V{V}}$.
	Apply Theorem~\ref{theo:induction-step} with the polynomial map 
	$$f = \D^a M \colon \FF^n \to A \otimes B \otimes (\FF^n)^{\otimes a} =: \TS{V}.$$
	As $H' \colon \FF^n \dashrightarrow \C{s}{\TS{V}}$ satisfies 
	$f|_{\V{V}} = (\D^a M)|_{\V{V}} = \IP\circ H'|_{\V{V}}$,
	it follows from this application that there exists a rational map 
	$H'' \colon \FF^n \dashrightarrow \C{s'}{\TS{V}\otimes\FF^n}$
	defined at $\B{x_0}$ such that
	$$(\D^{a+1} M)|_{\V{V}} 
	= (\D f)|_{\V{V}} = \IP\circ H''|_{\V{V}}$$
	with $s'=\max\{s,\codim\V{V}\}=s$. 
	This completes the induction step and the proof.
\end{proof}


\begin{remark}\label{remark:boring}
	It is instructive to consider the special case of Theorem~\ref{theo:induction-step} where $f \colon \FF^n \to \TS{V}$ vanishes on all of the variety $\V{V} \sub \overline{\FF}^n$. 
	For example, this happens for $f=\D\sl{T}\colon \FF^{n} \to \FF^m \otimes \FF^n$, where $\sl{T} \colon \prod_{i=1}^{k-1} \FF^{n_i}\to\FF^{m}$ ($n:=\sum_{i=1}^{k-1 }n_i$) is a $k$-tensor slicing,
	and $\V{V} = \{(\B{0},\B{0},\B{x}_3,\ldots,\B{x}_{k-1}) \mid \B{x}_i \in \overline{\FF}^{n_i}\} \sub \overline{\FF}^n$;
	indeed, this is because each entry of $f$ is a $(k-2)$-linear form on $\prod_{i \neq j} \FF^{n_i}$ for some $1 \le j \le k-1$, which necessarily vanishes on $\V{V}$.
	As we will see next, in the special case where $\V{V}$ is such that the polynomial $d$-tensor $f$ satisfies $f|_{\V{V}}=0$,
	Theorem~\ref{theo:induction-step}
	gives a partition rank decomposition of the $(d+1)$-tensor $\D f$ on an open subset of $\V{V}$ with at most $\codim\V{V}$ summands, and in fact, a flattening rank decomposition. 
	(The flattening rank of a tensor in $\bigotimes_{i=1}^k V_i$ is its matrix rank when viewed as a matrix in $(\bigotimes_{i\neq j} V_i) \otimes V_j$ for some $j$.)
	
	
	Note that $f|_{\V{V}}=0$ implies a trivial rational decomposition for $f$ on $\V{V}$, given by
	$f|_{\V{V}} = \IP \circ H|_{\V{V}}$ where $H \colon \FF^n \to \C{0}{\TS{V}}$ is simply $H=0$.
	Nevertheless, the proof of Theorem~\ref{theo:induction-step} guarantees a (generally non-trivial) rational decomposition for $\D f$ on $\V{V}$ of the form $(\D f)|_{\V{V}} = \IP \circ H'|_{\V{V}}$ with $H' \colon \FF^n \to \C{\codim\V{V}}{\TS{V} \otimes \FF^n}$.
	To see this, first note that since $H=0$, 
	the derivative of $f$ along any tangent vector to $\V{V}$ vanishes (see~(\ref{eq:D-tangent})).
	The proof constructs a rational map $Q \colon \FF^n \dashrightarrow \FF^n \otimes \FF^n$ such that 
	$Q(\B{x})$ is a projection matrix onto a complementary subspace of $\T_{\B{x}}\V{V}$ for every $\B{x} \in \V{V} \cap \domain(Q)$ (in particular, $\B{x}$ is non-singular on $\V{V}$),
	and thus obtains the decomposition (see~(\ref{eq:D-decomposition}))
	$$\D f(\B{x})
	= \sum_{i=1}^{\codim\V{V}} (\D f(\B{x}))Q(\B{x})_i \otimes Q^\tp(\B{x})_i.$$
	Since $(\D f(\B{x}))Q(\B{x})_i \in \TS{V}$ is a $d$-tensor and $Q^\tp(\B{x})_i \in \FF^n$ is a $1$-tensor,
	this is a flattening rank decomposition of the $(d+1)$-tensor $\D f(\B{x}) \in \TS{V} \otimes \FF^n$ for each $\B{x}$ in the open subset $\V{V} \cap \domain(Q)$.
	Let us finally note that iterating Theorem~\ref{theo:induction-step}, as in Corollary~\ref{coro:induction-step}, only gives a flattening rank decomposition for $(\D^{i+1} f)|_\B{V}$ as long as $(\D^{i} f)|_\B{V} = 0$; once $(\D^{i} f)|_\B{V} \neq 0$, the main term in the decomposition of $(\D^{i+1} f)|_\B{V}$ obtained by the proof may be nonzero, yielding in general a partition rank decomposition.
\end{remark}

\begin{proof}[Proof of Theorem~\ref{theo:PR-bound}]
	Let $T \in \bigotimes_{i \in [k]} V_i$ be a $k$-tensor over a perfect field $\FF$.
	Identify $V_1\times\cdots\times V_{k-1}$ with $\FF^{n}$, and $V_k$ with $\FF^m$.
	Let $\sl{T} \colon \FF^n \to \FF^{m}$ be the slicing of $T$, 
	and consider the polynomial matrix $\D \sl{T} \colon \FF^n \to \FF^m \otimes \FF^n$.
	We first prove that $\V{X} \sub \ker\sl{T}$ implies that for every $\B{x} \in \V{X}$,
	\begin{equation}\label{eq:T-vs-D}
	\rank(\D \sl{T}(\B{x})) \le \codim \V{X}.
	\end{equation}
	Write $\sl{T}=(\sl{T}_1,\ldots,\sl{T}_m)$, so that the $(k-1)$-linear forms $\sl{T}_1,\ldots,\sl{T}_m$ cut out $\ker \sl{T}$.
	This means that $\{\sl{T}_1,\ldots,\sl{T}_m\} \sub \I(\ker \sl{T}) \sub \I(\V{X})$. 
	We claim that $\T_\B{x}(\V{X}) \sub \ker(\D \sl{T}(\B{x}))$;
	indeed, for every $\B{x} \in \V{X}$, if $\B{v} \in \T_\B{x}(\V{X})$ then $(\D g(\B{x}))\B{v}=0$ for every $g \in \I(\V{X})$, and in particular for every $\sl{T}_i$, so 
	$$(\D\sl{T}(\B{x}))\B{v} = (\D \sl{T}_1(\B{x}),\ldots,\D \sl{T}_m(\B{x}))\B{v}
	= \B{0},$$
	meaning $\B{v} \in \ker(\D\sl{T}(\B{x}))$.
	We obtain~(\ref{eq:T-vs-D}) since
	$$
	\rank(\D \sl{T}(\B{x})) = n-\dim\ker(\D \sl{T}(\B{x})) \le n-\dim\T_\B{x}(\V{X}) \le \codim \V{X}
	$$
	where the last inequality applies Fact~\ref{fact:dim-tangent} using the irreducibility of $\V{X}$. 
	
	Put $M=\D \sl{T} \colon \FF^n \to \FF^m \otimes \FF^n$,
	and $r=\rank M(\B{x_0})$ where $\B{x_0}$ is an $\FF$-rational point on $\V{X}$ that is neither singular nor $M$-singular, as guaranteed by the statement.
 	The latter means that $r=\max_{\B{x} \in \V{X}} \rank M(\B{x})$; or, put differently, $M(\V{X}) \sub \V{M}_r$ where $\V{M}_r \sub \overline{\FF}^m \otimes \overline{\FF}^n$ is the set of matrices of rank at most $r$.
	Apply Lemma~\ref{lemma:rank-factorization} with $m$, $n$, and $r$, and use appropriate row and column permutations,
	to obtain a rational map $\R{F} \colon \FF^{m \times n} \dashrightarrow \FF^{m \times r} \times \FF^{r \times n}$ 
	that is defined at $M(\B{x_0})$ (using $\rank M(\B{x_0})=r$), such that for any matrix $A \in \V{M}_r \cap \domain(\R{F})$ we have that
	$\R{F}(A)=(\R{F}_1(A),\R{F}_2(A))$ is a rank factorization of $A$.
	Let $F \colon \FF^{n} \dashrightarrow \FF^{m \times r} \times \FF^{r \times n}$ 
	be the rational map given by the composition $F = \R{F} \circ M$.
	Note that $\domain(F)=\domain(\R{F})$.
	Thus, $F$ is defined at $\B{x_0}$, and for any $\B{x} \in M^{-1}(\V{M}_r) \cap \domain(F)$ we have that $F(\B{x})=(F_1(\B{x}),F_2(\B{x}))$ satisfy 
	$$M(\B{x}) = F_1(\B{x})F_2(\B{x}) 
	= \sum_{i=1}^r F_1(\B{x})_i \otimes F_2^\tp(\B{x})_i$$
	where $A_i$ denotes column $i$ of a matrix $A$.
	Let 
	$H(\B{x}) := (F_1(\B{x})_i,\, F_2^\tp(\B{x})_i)_{i \in [r]}$.
	Observe that for every $\B{x} \in \FF^n$ we have that $H(\B{x}) \in (\FF^m \times \FF^n)^r = \C{r}{\FF^m \otimes \FF^n}$.
	Note that $\domain(H)=\domain(F)$. 
	Thus, $H \colon \FF^n \dashrightarrow \C{r}{\FF^m \otimes \FF^n}$ is a rational map defined at $\B{x_0}$ such that 
	for every $\B{x} \in M^{-1}(\V{M}_r) \cap \domain(H)$ we have
	$M(\B{x}) = \IP(H(\B{x}))$.
	Since $\V{X} \sub M^{-1}(\V{M}_r)$, 
	this means that
	$$M|_\V{X} = \IP \circ H|_\V{X} .$$
	
	Now, apply Corollary~\ref{coro:induction-step} with the irreducible variety $\V{X} \sub \overline{\FF}^n$, the non-singular point $\B{x_0} \in \V{X}$, the polynomial matrix $M\colon \FF^n \to \FF^m \otimes \FF^n$, the rational map $H \colon \FF^n \dashrightarrow \C{r}{\FF^m \otimes \FF^n}$, and $a=k-2$,
	to obtain a rational map $H' \colon \FF^n \dashrightarrow \C{s}{\FF^m \otimes (\FF^n)^{\otimes k-1}}$ that is defined at $\B{x_0}$ such that 
	$\D^{k-2} M|_{\V{X}} = \IP\circ H'|_{\V{X}}$
	for $s = \max\{r,\,\codim \V{X}\} = \codim \V{X}$
	by~(\ref{eq:T-vs-D}).
	Evaluating at $\B{x_0} \in \V{X}$, we obtain 
	$$\D^{k-1} \sl{T}(\B{x_0}) = \D^{k-2} M(\B{x_0}) = \IP(H'(\B{x_0})) .$$
	Since $\B{x_0}$ is $\FF$-rational, this implies that 
	$\D^{k-1} \sl{T}(\B{x_0}) \in \IP(\C{\codim\V{X}}{\FF^m \otimes (\FF^n)^{\otimes k-1}})$.
	Therefore,
	$$\PR(T) \le \PR(\D^{k-1} \sl{T}(\B{x_0})) \le (2^{k-1}-1)\codim\V{X},$$
	where the first inequality follows from Proposition~\ref{claim:PR-restriction}
	and the second inequality from Fact~\ref{fact:PR-image}.
	This completes the proof.
\end{proof}

\subsection*{Technical comparison with Schmidt's proof.\,} 
Here we expand on several key differences between Schmidt's proof~\cite{Schmidt84,Schmidt85}, which gives a version of Corollary~\ref{theo:main2}, and ours.
Schmidt constructs certain functions\footnote{These are referred to as global sections in~\cite{KazhdanPo21}.} which enable expressing $T$ using  bilinear forms (Section~22 in~\cite{Schmidt85}),  
and uses their derivatives---as well as additional linear forms---to obtain 
a decomposition for $T$
in a recursive manner.
In contrast, we take derivatives of $T$ itself, which are then restricted to an open subset of a subvariety of $\ker\sl{T}$. 
The core analysis in Schmidt's proof (Sections~22 and~23 in~\cite{Schmidt85}) significantly differs from ours, and in particular leads to an asymptotically worse multiplicative constant ($\PR(T) \le C_k\GR(T)$ with $C_k$ roughly $k!$, as opposed to $2^{k-1}-1$ in our Corollary~\ref{theo:main2}).
Additionally, unlike in~\cite{Schmidt85},
the point $\B{x_0}$ at which we take derivatives (recall Theorem~\ref{theo:induction-step}) is required to
be not just non-singular but also $\D\sl{T}$-non-singular.
Furthermore, we obtain $\B{x_0}$ using an iterative process that is more subtle, and requires more machinery; this includes defining a variant of the notion of degree of a variety that is suitable for our needs (see Definition~\ref{def:deg-growth} below).
We control all irreducible components throughout the process, and in particular retain all the rational points. 
Somewhat counter-intuitively, $\B{x_0}$ is ``generic'' only on a subvariety of $\ker\sl{T}$ which might be considerably smaller, 
in which case we do not expect to glean much information on $T$ from its derivatives at $\B{x_0}$.
In fact, this subvariety could end up being the zero subvariety $\{\B{0}\}$, and so $\B{x_0}=\B{0}$, in which case the repeated derivatives of the multilinear $T$ at $\B{x_0}$ reveal no information on $T$ at all! (Recall Remark~\ref{remark:boring}.) However, since the subvariety is constructed in such a way that it contains all the $\FF$-rational points of $\ker\sl{T}$, in this extreme case we also have that $\AR(T)$ is maximal, and so there is nothing to prove.

\section{Proof of the Main Theorem}\label{sec:final}
In this section we show that any variety has all its $\FF$-rational points contained in a subvariety $\V{Z}$ of controlled complexity and, moreover, at least one of those $\FF$-rational points---if any exist---satisfies a certain strong notion of non-singularity with respect to a top-dimensional irreducible component of $\V{Z}$ defined over $\FF$. 
We then show that for the $\ker\sl{T}$ variety, a subvariety as above yields a lower bound on $\AR(T)$ which matches the upper bound in Theorem~\ref{theo:PR-bound}, thus completing the proof of Theorem~\ref{theo:main}.


\subsection{Auxiliary results}



For a variety $\V{V} \sub \overline{\FF_q}^n$, where $\FF_q$ is the finite field of size $q$, 
we denote the image of the Frobenius automorphism on $\V{V}$ by 
$\Frob(\V{V}) = \{ (x_1^q,\ldots,x_n^q) \mid (x_1,\ldots,x_n) \in \V{V}\}$.
The Frobenius automorphism characterizes definability over $\FF_q$, as follows (see Tao's Corollary 4 in~\cite{Tao12}; we include a proof for completeness). 
\begin{fact}\label{lemma:Frob}
	Let $\FF$ be a finite field, and let $\V{V}$ be 
	a variety
	over $\overline{\FF}$.
	Then $\V{V}$ is defined over $\FF$ if and only if $\Frob(\V{V})=\V{V}$.
%
\end{fact}
\begin{proof}[Proof]
	Put $\FF=\FF_q$, and write $\B{x}^q$ for $(x_1^q,\ldots,x_n^q)$.
	For an $\overline{\FF}$-polynomial $f$, let $\Frob(f)$ be obtained from $f$ by replacing each coefficient $c$ with $c^q$. 
	Observe
	that $\Frob(f)(\B{x}^q) = f(\B{x})^q$,
	and that $\Frob(f)=f$ if and only if $f$ is an $\FF$-polynomial (in which case $f(\B{x}^q) = f(\B{x})^q$).
	
	Suppose $\V{V}$ is defined over $\FF$.
	Then $\B{x}^q \in \V{V}$ iff $\B{x} \in \V{V}$ iff $\B{x}^q \in \Frob(\V{V})$.
	Since every $\B{y} \in \overline{\FF}^n$ is of the form $\B{x}^q$ for some $\B{x} \in \overline{\FF}^n$, we get $\Frob(\V{V})=\V{V}$.

	Suppose $\Frob(\V{V})=\V{V}$.
	For any $f \in \I(\V{V}):=I$ we have $\Frob(f) \in I$; 
	indeed, $\Frob(f)(\B{x}^q) = f(\B{x})^q = 0$ for $\B{x} \in \V{V}$, so $\Frob(f) \in \I(\Frob(\V{V})) = I$.
	Let $G \sub I$ be the reduced Gr\"{o}bner basis of $I$, under any arbitrary monomial ordering
	(see, e.g.,~\cite{CoxLiOS}, Chapter~2.7, Theorem~5).
	Replacing any $f \in G$ with $\Frob(f) \in I$ again yields a reduced Gr\"{o}bner basis of $I$, 
	since $\Frob(f)$ has the same support as $f$ and its leading coefficient is $1$ as well. The uniqueness of $G$ 
	implies $f=\Frob(f)$, so $f$ is an $\FF$-polynomial. Thus, $\I(\V{V}) = \langle G \rangle$ is generated by $\FF$-polynomials, completing the proof.
\end{proof}






The \emph{degree} $\deg\V{V}$ of an irreducible variety $\V{V} \sub \overline{\FF}^n$ is the cardinality of the intersection of $\V{V}$ with a generic linear subspace in $\overline{\FF}^n$ of dimension $\codim\V{V}$ (a well-defined number).
The \emph{(total) degree} of an arbitrary variety is the sum of the degrees of its irreducible components.
%
The following is a version of the Schwartz-Zippel lemma for varieties (see, e.g.,~\cite{BukhTs12,DvirKoLo14}).

\begin{fact}[Generalized Schwartz-Zippel lemma]\label{fact:SZ}
	Let $\FF$ be a finite field.
	For any variety $\V{V}$ over $\overline{\FF}$,
	$$|\V{V}(\FF)| \le \deg\V{V}\cdot|\FF|^{\dim\V{V}}.$$
\end{fact}


The following is a version of B\'{e}zout’s theorem (see~\cite{Fulton} p.~223, Example~12.3.1, or~\cite{BukhTs12}).

\begin{fact}[Generalized B\'{e}zout’s theorem]\label{fact:Bezout}
	For any two varieties $\V{U},\V{V} \sub \overline{\FF}^n$,
	$$\deg(\V{U} \cap \V{V}) \le \deg\V{U}\cdot\deg\V{V}.$$
\end{fact}

Let $\V{V}_0$ denote the union of the $0$-dimensional irreducible components of a variety $\V{V}$. 
We will also use the following overdetermined case of B\'{e}zout's theorem (see~\cite{Tao11}, Theorem 5).
\begin{fact}[B\'{e}zout's theorem, overdetermined case]\label{fact:bezout-over}
	Let $\V{V} \sub \overline{\FF}^n$ be a variety
	cut out by $m \ge n$ polynomials $f_1,\ldots,f_m$. 
	Write $\deg f_1 \ge \cdots \ge \deg f_m \ge 1$. 
	Then 
	$$|\V{V}_0| \le \prod_{i=1}^n \deg f_i.$$
\end{fact}


%

We henceforth denote by $\Z(f)$ the variety cut out by a (single) polynomial $f$.
The following lemma originates with Mumford's paper~\cite{Mumford70}, and can be deduced from an argument in Catanese's paper~\cite{Catanese92} (see also~\cite{BlancoJeSo04} and the references within). For completeness, we sketch Catanese's proof below.
\begin{fact}[\cite{Mumford70},\cite{Catanese92}]\label{fact:Mumford}
	Let $\V{X} \sub \overline{\FF}^n$ be an irreducible variety, and let $\B{x} \in \V{X}$ be a non-singular point. 
	Then there are $r=\codim\V{X}$ polynomials $f_1,\ldots,f_r \in \I(\V{X})$ of degrees at most $\deg\V{X}$ such that $\T_{\B{x}} \V{X} = \T_{\B{x}} \Z(f_1) \cap \dots \cap \T_{\B{x}} \Z(f_r)$. 
\end{fact}

\begin{proof}[Proof sketch]
	Put $k:=\dim \V{X}+1$ and $\KK:=\overline{\FF}$.
	Call a linear map $p: \KK^n \to \KK^k$ \emph{nice} if 
	$p|_\V{X} \colon \V{X} \to \overline{p(\V{X})}$ is birational, and locally an isomorphism at $\B{x}$, in the sense that there is an open neighborhood $U \sub \V{X}$ of $\B{x}$ such that $p|_U \colon U \to p(U)$ is an isomorphism. 
	The proof follows as in~\cite{Catanese92} (paragraph preceding Remark~(1.17)):
	\begin{enumerate}
		\item The nice maps form a nonempty open set in the space of all linear maps from $\KK^n \to \KK^k$ (the proof uses Chow forms).
		\item\label{item:proj-tang} For every nice map $p$ there is a polynomial $f$ with $\deg f \le \deg \V{X}$ satisfying $\T_\B{x} \V{X} = \T_\B{x} \Z(f) + \ker p$.
		\item Hence, for a generic choice of $r=\codim\V{X} = \codim\T_\B{x} \V{X}$ nice maps $p_1, \ldots, p_r$, the resulting polynomials $f_1, \ldots, f_r$ satisfy $\T_{\B{x}} \V{X} = \T_{\B{x}} \Z(f_1) \cap \dots \cap \T_{\B{x}} \Z(f_r)$, as needed.
	\end{enumerate}
	
	
	For completeness, we include a proof of~(\ref{item:proj-tang}).
	Let $p$ be nice. Then the variety $\V{X}^*:=\overline{p(\V{X})}$ is irreducible, $\dim \V{X}^* = \dim \V{X}$, 
	and $p$ restricts to an isomorphism between the linear subspaces $\T_\B{x} \V{X}$ and $\T_{p(\B{x})} \V{X}^*$.
	We deduce that $\V{X}^* \sub \KK^k$ is a hypersurface ($\codim_{\KK^k}\V{X}^*=1$), so $\V{X}^*=\Z(g)$ for some irreducible polynomial $g \colon 
	\KK^k \to \KK$ with $\deg g = \deg\V{X}^*$.
	We have
	$p(\T_\B{x} \V{X}) = \T_{p(\B{x})} \Z(g) = p(\T_\B{x} \Z(g \circ p))$, 
	where the last equality uses the definition of a tangent space of hypersurfaces:
	\begin{align*}
		\T_{p(\B{x})} \Z(g) &= \{ \B{u} \in \KK^k \mid (\D g(p(\B{x})))\B{u} = 0\}\\
		&= \{ p(\B{v}) \in \KK^k \mid (\D g(p(\B{x})))p(\B{v}) = 0\}\\
		&= p\big(\{ \B{v} \in \KK^n \mid (\D g\circ p (\B{x}))\B{v} = 0\}\big)
		= p(\T_\B{x} \Z(g \circ p)).
	\end{align*}
	Put $f:=g \circ p \colon \KK^n \to \KK$, and note that $\deg f \le \deg g = \deg \V{X}^* \le \deg \V{X}$ (for the last inequality see, e.g., Lemma~2 in~\cite{Heintz83}, or the proof of Lemma 3.4 in~\cite{DvirKoLo14}).
	Thus, $p(\T_\B{x} \V{X}) = p(\T_\B{x} \Z(f))$, which proves~(\ref{item:proj-tang}).
\end{proof}

For an irreducible variety $\V{X} \sub \overline{\FF}^n$ and a polynomial matrix $M \colon \overline{\FF}^n \to \overline{\FF}^{m \times n}$, 
the $M$-singular locus of $\V{X}$ is the subvariety of $M$-singular points, 
$$\Sing_M\V{X} = \Big\{ \B{x'} \in \V{X} \,\Big\vert\, \rank M(\B{x'}) < \max_{\B{x} \in \V{X}} \rank M(\B{x}) \Big\}.$$
If $M$ is the Jacobian matrix of a generating set of $\I(\V{X})$ then $\Sing_M$ is the singular locus of $\V{X}$, in which case we omit the subscript $M$.
Note that $\Sing_M\V{X}$ is a strict subvariety of $\V{X}$.
%
We next show that the $M$-singular locus is contained in a strict subvariety of controlled degree.

\begin{lemma}\label{lemma:Sing-deg-bd}
	Let $\V{X} \sub \overline{\FF}^n$ be an irreducible variety, 
	let $M \colon \overline{\FF}^n \to \overline{\FF}^{m \times n}$ be a polynomial matrix,
	and put $r=\max_{\B{x} \in \V{X}} \rank M(\B{x})$.
	Suppose there is an $M$-non-singular point $\B{x} \in \V{X}$
	witnessed by an $r \times r$ submatrix of $M$ (i.e., invertible at $\B{x}$) whose every polynomial is of degree at most $d$.
	Then there is a strict subvariety $\Sing_M\V{X} \sub \V{Y} \subsetneq \V{X}$
	with $\deg\V{Y} \le rd\deg\V{X}$.
\end{lemma}
\begin{proof}
	Let $D \colon \overline{\FF}^n \to \overline{\FF}$ be the determinant of the $r \times r$ submatrix of $M$ in the statement.
	Note that $D$ is a polynomial of degree at most $rd$ satisfying $D(\B{x}) \neq 0$. 
	Let $\V{Y} = \Z(D) \cap \V{X}$.
	We have $\Sing_M\V{X} \sub \V{Y} \subsetneq \V{X}$ since 
	every $M$-singular point $\B{x'}$ on $\V{X}$ satisfies $D(\B{x'})=0$ as $\rank M(\B{x'}) < r$, 
	and since $\B{x} \notin \V{Y}$. Moreover, by the B\'{e}zout's bound in Fact~\ref{fact:Bezout}, $\deg\V{Y} \le \deg\Z(D) \cdot \deg\V{X} \le rd\deg\V{X}$. This completes the proof.
\end{proof}

Fact~\ref{fact:Mumford} and Lemma~\ref{lemma:Sing-deg-bd} together yield a degree bound for the (standard) singular locus.\footnote{This also likely follows from a covering argument of Helfgott in~\cite{Helfgott18}.}
\begin{corollary}\label{coro:Sing-deg-bd}
	Every irreducible variety $\V{X} \sub \overline{\FF}^n$ has a strict subvariety $\Sing\V{X} \sub \V{Y} \subsetneq \V{X}$
	with 
	$\deg\V{Y} \le \codim\V{X}(\deg\V{X}-1)\deg\V{X}$.
\end{corollary}
\begin{proof}
	Let $\B{x}$ be any non-singular point on $\V{X}$. 
	Consider the $r:=\codim\V{X}$ polynomials of degree at most $\deg\V{X}$ guaranteed by Fact~\ref{fact:Mumford}, and let $J \colon \overline{\FF}^n \to \overline{\FF}^{r \times n}$ be their Jacobian matrix, so that $\ker J(\B{x}) = \T_{\B{x}} \V{X}$.
	Then $\rank J(\B{x}) = \codim\T_{\B{x}} \V{X} = \codim\V{X} = r$ 
	by our choice of $\B{x}$.
	Apply Lemma~\ref{lemma:Sing-deg-bd} on $\V{X}$, $J$, 
	and any invertible $r \times r$ submatrix of $J(\B{x})$
	(note $\max_{\B{x'} \in \V{X}} \rank J(\B{x'}) = \codim\V{X}=r$ by Fact~\ref{fact:dim-tangent}).
	Since every polynomial in $J$ is of degree at most $\deg\V{X}-1$, we thus obtain a strict subvariety $\Sing\V{X} \sub \V{Y} \subsetneq \V{X}$
	with $\deg\V{Y} \le r(\deg\V{X}-1)\deg\V{X}$, completing the proof.
	%
%
\end{proof}



\subsection{Main proof}

For a variety $\V{V}$, let $\V{V}^\top$ denote the union of the irreducible components of $\V{V}$ of dimension $\dim\V{V}$, and let $\V{V}_\bot$ denote the union of the remaining, lower-dimensional irreducible components.

We define a variant of the notion of the degree of a variety
which would allow us to iterate complexity bounds.
For a variety $\V{V} \sub \overline{\FF}^n$, let $\V{V}_m$ denote its codimension-$m$ part, meaning the union of all irreducible components of $\V{V}$ of codimension $m$.

\begin{definition}[Degree-growth]\label{def:deg-growth}
	For a variety $\V{V} \sub \overline{\FF}^n$, 
	$$\d(\V{V}) = \max_{m} \deg (\V{V}_m)^{1/m}. $$
\end{definition}
By convention we set $\d\big(\overline{\FF}^n\big)=0$.

\begin{claim}\label{claim:props}
We have the following properties for arbitrary varieties $\V{U}$, $\V{V}$:
\renewcommand{\theenumi}{\roman{enumi}}
\begin{enumerate}
	\item\label{item:prop-d-bounds} $\deg(\V{V}^\top) \le \d(\V{V})^{\codim\V{V}} \le \deg(\V{V})$.\\
	In particular, if $\V{V}$ is equidimensional\footnote{That is, all irreducible components are of the same dimension.} (or irreducible) then $\deg(\V{V})=\d(\V{V})^{\codim\V{V}}$.
	\item\label{item:prop-d-poly} $\d(\V{V}) \le d$ if $\V{V}$ is cut out by polynomials of degrees at most $d$.
	\item\label{item:prop-d-union} $\d(\V{U} \cup \V{V}) \le \d(\V{U})+\d(\V{V})$.
\end{enumerate}
\end{claim}
We note the Property~(\ref{item:prop-d-poly}) is not satisfied by either $\deg(\V{V})$ or $\deg(\V{V}^\top)$, which are the notions of complexity of a variety frequently used in the literature, since neither the irreducible components nor equidimensional components of $\V{V}$ are generally cut out by polynomials of degree also at most $d$.
In the proof of Theorem~\ref{theo:main} we will use Property~(\ref{item:prop-d-poly}) to bound $\d(\ker\sl{T})$, and use Property~(\ref{item:prop-d-bounds}) to iterate that bound for the subvarieties we construct. 

\begin{proof}[Proof of Claim~\ref{claim:props}]
Property~(\ref{item:prop-d-bounds}) follows from the inequalities
$$\deg(\V{V}_{\codim\V{V}})^{1/\codim\V{V}} \le \d(\V{V}) \le \max_{m \ge \codim\V{V}} \deg(\V{V})^{1/m} ,$$
where the last inequality uses $\deg(\V{V}_m) \le \deg(\V{V})$.
Property~(\ref{item:prop-d-poly}) follows by noting that for each $m$, intersecting $\V{V}_m$ with a generic linear subspace in $\overline{\FF}^n$ of dimension $m$ gives a $0$-dimensional variety---naturally embedded in $\overline{\FF}^m$ and cut out by polynomials of degrees at most $d$---whose cardinality is $\deg\V{V}_m \le d^m$ by Fact~\ref{fact:bezout-over}.
(Property~(\ref{item:prop-d-poly}) can also be proved along the lines of Proposition~2.3 in~\cite{Faltings91}.)
Property~(\ref{item:prop-d-union}) follows from, for some $m$:
\begin{align*}
	\d(\V{U} \cup \V{V}) 
	\le (\deg \V{U}_m + \deg \V{V}_m)^{1/m} 
	\le (\deg \V{U}_m)^{1/m} + (\deg \V{V}_m)^{1/m} 
	\le \d(\V{U}) + \d(\V{V}) .
\end{align*}
\end{proof}
%
%
%
%
We further have the following version of the Schwartz-Zippel lemma for 
degree-growth. 

\begin{lemma}\label{lemma:SZ-d}
	Let $\FF$ be a finite field.
	For every variety $\V{V} \sub \overline{\FF}^n$ with $\d(\V{V})<|\FF|$,
	$$|\V{V}(\FF)| \le \frac{d(\V{V})^{\codim\V{V}}|\FF|^{\dim\V{V}}}{1-\d(\V{V})/|\FF|}.$$
	%
	%
\end{lemma}
\begin{proof}
	Put $\delta=\d(\V{V})/|\FF|$. 
	Using Fact~\ref{fact:SZ},
	\begin{align*}
		\frac{|\V{V}(\FF)|}{|\FF|^n}
		&\le \sum_{m \ge \codim\V{V}} \deg\V{V}_m|\FF|^{-m}  
		\le \sum_{m \ge\codim\V{V}} \delta^m
		\le \delta^{\codim\V{V}}/(1-\delta),
	\end{align*}
	where the second inequality bounds $\deg(\V{V}_m) \le \d(\V{V})^m$.
	This completes the proof.
\end{proof}

We are now ready to prove the main theorem.
Our proof produces a nested sequence of subvarieties, obtained by applying---separately on each top-dimensional irreducible component---either one of two operators: one intersects the component with its image under the Frobenius automorphism, and the other takes a low-degree strict subvariety that contains a generalized singular locus.
We show that if any irreducible component is fixed by both operators then it must satisfy the properties needed for Theorem~\ref{theo:PR-bound}, thus bounding the partition rank from above. Moreover, we show that each constructed subvariety retains all the $\FF$-rational points and is of controlled complexity, which enables bounding the analytic rank from below and completing the proof of Theorem~\ref{theo:main} (note that we apply Theorem~\ref{theo:PR-bound} on an irreducible component, which need not retain all $\FF$-rational points!).

\begin{proof}[Proof of Theorem~\ref{theo:main}]
Let $T$ be a $k$-tensor over a finite field $\FF$,
let $\sl{T} \colon \FF^n \to \FF^m$ be the slicing of $T$,
and put $M = \D\sl{T} \colon \FF^n \to \FF^m \otimes \FF^n$.
For an irreducible variety $\V{X}$ over $\overline{\FF}$, 
let $S_M(\V{X})$ be the strict subvariety from Lemma~\ref{lemma:Sing-deg-bd}, 
let $S(\V{X})$ be the strict subvariety from Corollary~\ref{coro:Sing-deg-bd}, 
and let $S^+(\V{X}) = S(\V{X}) \cup S_M(\V{X})$.
Note that $S^+(\V{X})$ is thus also a strict subvariety of $\V{X}$.
We define the following operators mapping any irreducible variety $\V{X}$ over $\overline{\FF}$ to a subvariety thereof:
\begin{itemize}
	\item $\F(\V{X}) = \V{X} \cap \Frob\V{X}$.
	\item $\Si(\V{X}) = S^+(\V{X})$ if $\V{X}(\FF) \sub S^+(\V{X})$, and otherwise $\Si(\V{X})=\V{X}$.
	\item $\Op(\V{X})=\F(\V{X})$ if $\F(\V{X})\neq\V{X}$, otherwise $\Op(\V{X})=\Si(\V{X})$.
\end{itemize}
For an arbitrary variety $\V{V}$, we define $\Op(\V{V}) = \bigcup_{\V{X}} \Op(\V{X}) \cup \V{V}_\bot$ where the union is over all irreducible components $\V{X}$ of $\V{V}^\top$.
We will show that $\Op$ satisfies the following properties 
for any variety $\V{V} \sub \ker\sl{T}$ (this assumption on $\V{V}$ enables finding a $\D\sl{T}$-non-singular point):
\renewcommand{\theenumi}{\arabic{enumi}}
\begin{enumerate}
	\item\label{item:prop-preserve} $\V{V}(\FF) \sub \Op(\V{V})$.
	\item\label{item:prop-deg} 
	$\d(\Op(\V{V})) \le 2\f_{\V{V}}(\d(\V{V}))$ where
	$\f_{\V{V}}(x) = 2k\codim\V{V} \cdot x^2$.
	%
	%
	\item\label{item:prop-good} If $\dim\Op(\V{V})=\dim\V{V}$ and $\V{V}\neq\emptyset$ then $\V{V}^\top$ has an irreducible component defined over $\FF$ that has an $\FF$-rational point that is neither singular nor $M$-singular.
\end{enumerate}
We prove these properties in order:
\begin{enumerate}
	\item For any irreducible variety $\V{X}$ we have $\V{X}(\FF) \sub \F(\V{X})$ (since $\Frob$ is an automorphism on $\FF^n$) and $\V{X}(\FF) \sub \Si(\V{X})$ (by construction), which implies
	$\V{X}(\FF) \sub \Op(\V{X})$ and thus $\V{V}(\FF) \sub \Op(\V{V})$ for arbitrary $\V{V}$.
	\item First we prove the degree bounds $\deg(\F(\V{X})) \le \f_{\V{X}}(\deg\V{X})$ and $\deg(\Si(\V{X})) \le \f_{\V{X}}(\deg\V{X})$
	for any irreducible variety $\V{X} \sub \ker\sl{T}$.
	For $\F$, 
	note that the Frobenius automorphism $\Frob$ maps hyperplanes to hyperplanes. As a consequence, for any irreducible variety $\V{X} \sub \overline{\FF}^n$, the variety $\Frob(\V{X})$ is isomorphic to $\V{X}$ and has the same degree; indeed, for $L\sub \overline{\FF}^n$ a generic linear subspace of dimension $\codim\V{X}$, 
	$$\deg(\Frob(\V{X})) = |\Frob(\V{X}) \cap L| = |\V{X} \cap \Frob^{-1}(L)| = \deg(\V{X}).$$
	Thus, by B\'{e}zout's bound in Theorem~\ref{fact:Bezout}, 
	$$\deg(\F(\V{X})) \le \deg(\V{X} \cap \Frob\V{X}) \le \deg(\V{X}) \cdot \deg(\Frob\V{X}) 
	= \deg(\V{X})^2 
	\le \f_{\V{X}}(\deg\V{X})$$
	(if $\codim\V{X}=0$ there is nothing to prove).
	For $\Si(\V{X})$, 
	we have by Lemma~\ref{lemma:Sing-deg-bd} and Corollary~\ref{coro:Sing-deg-bd} that
	\begin{align*}
		\deg\Si(\V{X}) &\le \deg S(\V{X}) + \deg S_M(\V{X}) \le 
		\codim\V{X}(\deg\V{X}-1)\deg\V{X} + \codim\V{X}(k-2)\deg\V{X} \\
		&\le \codim\V{X}(\deg\V{X}+k)\deg\V{X}
		\le  \codim\V{X} \cdot 2k(\deg\V{X})^2
		= \f_{\V{X}}(\deg\V{X}),
	\end{align*}
	using, to bound $\deg S_M(\V{X})$, the fact that each entry in the polynomial matrix $M=\D\sl{T}$ is a polynomial of degree $k-2$, as well as using the bound $\max_{\B{x} \in \V{X}} \rank M(\B{x}) \le \codim\V{X}$ that follows from the assumption $\V{X} \sub \ker\sl{T}$ (recall~(\ref{eq:T-vs-D})),
	and using, for the last inequality, the bound $(x+y)/2 \le \max\{x,y\} \le xy$ for reals $x,y\ge 1$. 
	Now, to prove the desired degree bound for $\Op(\V{V})$, 
	first note that for an irreducible variety $\V{X} \sub \ker\sl{T}$ we have, by construction and the above bounds, that
	$\deg \Op(\V{X}) \le \f_{\V{X}}(\deg\V{X})$.
	An arbitrary variety $\V{V} \sub \ker\sl{T}$ requires more care;
	first, we prove the bound $\d(\Op(\V{V}^\top)) \le \f_{\V{V}}(\d(\V{V}))$; 
	it follows from
	\begin{align*}
		\d(\Op(\V{V}^\top))^{\codim\Op(\V{V}^\top)} 
		&\le \deg\Op(\V{V}^\top)
		= \sum_\V{X} \deg\Op(\V{X})
		\le \sum_\V{X} \f_{\V{X}}(\deg\V{X})
		= \sum_\V{X} \f_{\V{V}}(\deg\V{X})\\
		&\le \f_{\V{V}}(\deg(\V{V}^\top))
		\le \f_{\V{V}}(\d(\V{V})^{\codim\V{V}})
		\le \f_{\V{V}}(\d(\V{V}))^{\codim\V{V}}
	\end{align*}
	where the first inequality uses Property~(\ref{item:prop-d-bounds}) of $\d$, 
	the sum is over the irreducible components $\V{X}$ of $\V{V}^\top$,
	and the penultimate inequality uses Property~(\ref{item:prop-d-bounds}) again.
	We deduce that
	$$\d(\Op(\V{V})) = \d(\Op(\V{V}^\top) \cup \V{V}_\bot)
	\le \d(\Op(\V{V}^\top)) + \d(\V{V}_\bot)
	\le \f_{\V{V}}(\d(\V{V})) + \d(\V{V})
	\le 2\f_{\V{V}}(\d(\V{V}))$$
	where the first inequality uses Property~(\ref{item:prop-d-union}) of $\d$, 
	and the second inequality follows from the bound above as well as the definition of $\d$.
	\item Since $\Op$ maps varieties to subvarieties thereof, $\dim\Op(\V{V})=\dim\V{V}$ with $\V{V}\neq\emptyset$ implies that there is an irreducible component $\V{X}$ of $\V{V}^\top$ such that $\Op(\V{X}) = \V{X}$, which in turn implies that $\F(\V{X})=\V{X}$ and $\Si(\V{X})=\V{X}$.
	The former means that $\V{X}$ is defined over $\FF$ by Fact~\ref{lemma:Frob}, 
	and the latter means that $\V{X}(\FF) \nsubseteq S^+(\V{X})$
	and so $\V{X}(\FF) \nsubseteq \Sing\V{X} \cup \Sing_M\V{X}$ (recall $\Sing\V{X} \cup \Sing_M\V{X} \sub S^+(\V{X})$ by Lemma~\ref{lemma:Sing-deg-bd} and Corollary~\ref{coro:Sing-deg-bd}).
	Thus, $\V{X}$ satisfies the required conditions.
		%
\end{enumerate}

%

Now, let $\V{Z}=\Op^{(t)}(\ker\sl{T})$ be the $t$-times iteration of $\Op$ on $\ker\sl{T}$, where $t \ge 0$ is the smallest integer such that $\dim\Op^{(t+1)}(\ker\sl{T}) = \dim\Op^{(t)}(\ker\sl{T})$.
Note that $\V{Z} \neq \emptyset$ since, by Property~(\ref{item:prop-preserve}), we have $\B{0} \in \ker\sl{T}(\FF) \sub \V{Z}$.
%
Apply Theorem~\ref{theo:PR-bound} on 
the irreducible component of $\V{Z}^\top$ guaranteed by Property~(\ref{item:prop-good}).
This yields the following upper bound on $\PR(T)$:
\begin{equation}\label{eq:PR-bound}
\PR(T) \le (2^{k-1}-1)\codim\V{Z}.
\end{equation}
Next, we prove a lower bound on $\AR(T)$.
Observe that for any variety $\V{V}$, if $\ker\sl{T}(\FF) \sub \V{V}$ 
and $\d(\V{V}) < |\FF|$ then
\begin{align}
\begin{split}\label{eq:AR-deg}
\AR(T) &= n - \log_{|\FF|}|\V{\ker\sl{T}}(\FF)|
\ge n - \log_{|\FF|}|\V{V}(\FF)|
= \codim\V{V} - \log_{|\FF|}\frac{|\V{V}(\FF)|}{|\FF|^{\dim\V{V}}}\\
&\ge \codim\V{V}\big(1 - \log_{|\FF|}\d(\V{V})\big) + \log_{|\FF|}\Big(1-\frac{\d(\V{V})}{|\FF|}\Big)
\end{split}
\end{align}
where for the last inequality we apply Lemma~\ref{lemma:SZ-d}. 
%
We first claim that for any $C \ge 1$, 
\begin{align}\label{eq:AR-F-bd}
	\begin{split}
	&\text{if } \ker\sl{T}(\FF) \sub \V{V} \text{ and } 
	|\FF| \ge 4C\cdot\d(\V{V})^{2(C\cdot\AR(T)+1)}\\
	&\text{ then } \AR(T) \ge \codim\V{V} - 1/C.
	\end{split}
\end{align}
If $\log_{|\FF|}(1-\frac{\d(\V{V})}{|\FF|}) \ge -\frac{1}{2C}$
and $1-\log_{|\FF|}\d(\V{V}) \ge (\AR(T)+\frac{1}{2C})/(\AR(T)+\frac{1}{C})$ then~(\ref{eq:AR-deg}) immediately gives $\AR(T)+\frac{1}{C} \ge \codim\V{V}$.
The former condition follows from $|\FF| \ge 4C\cdot\d(\V{V})$ 
(using the inequality $\log_2(1-x) \ge -2x$ for $0 \le x \le 1/2$ with $x=\frac{1}{4C}$),
and the latter condition is equivalent to $|\FF| \ge \d(\V{V})^{2(C\AR(T)+1)}$,
thus verifying~(\ref{eq:AR-F-bd}).

%

%
%
%
%
%
%
%
%
%
%
%

Let $\V{Z}_i = \Op^{(i)}(\ker\sl{T})$ for each $i$.
For every $0 \le i < t$,
\begin{equation}\label{eq:Z-prop}
\dim(\V{Z}_{i+1}) < \dim(\V{Z}_{i}) \,\text{ and }\,
\d(\V{Z}_{i+1}) \le 2\f_{\V{Z}_i}(\d(\V{Z}_{i}))
\,\text{ and }\,
\ker\sl{T}(\FF) \sub \V{Z}_i ,
\end{equation}
where the first item follows immediately from the construction of $\V{Z}$,
the second item from Property~(\ref{item:prop-deg}) of $\Op$,
and the third item from Property~(\ref{item:prop-preserve}).
Let $A=\ceil{\AR(T)}$ and $C \ge 1$, 
and suppose
\begin{equation}\label{eq:F-size}
|\FF| \ge k^{CA^2 2^{A+8}}.
\end{equation}
We claim that for every $0 \le i \le t$,
$$\AR(T) \ge \codim\V{Z}_i - 1/C .$$
We proceed by induction on $i$. 
For the base case $i=0$, first note that $\d(\V{Z}_0) = \d(\ker\sl{T}) \le k-1$, using the fact that $\ker\sl{T}$ is cut out by the $(k-1)$-linear polynomials in the entries of $\sl{T}$ and, importantly, Property~(\ref{item:prop-d-poly}) of $\d$.
Thus, the induction base follows from~(\ref{eq:AR-F-bd}) and~(\ref{eq:F-size}).
%
For the induction step assume $\AR(T) \ge \codim\V{Z}_i - 1/C$ (which implies $\codim\V{Z}_i\le \AR(T)+1$), and observe that
$$\d(\V{Z}_{i+1})
\le 4k \cdot 2A \cdot \d(\V{Z}_i)^2
\le (8kA\cdot\d(\V{Z}_{i+1-j}))^{2^j}
\le (8k^2 A)^{2^{i+1}} 
\le k^{A 2^{A+5}} 
\le (|\FF|/4C)^{\frac{1}{2(CA+1)}}$$
where the first inequality follow from the second item of~(\ref{eq:Z-prop}), 
the second inequality applies for all $1 \le j \le i+1$,
the third inequality takes $j=i+1$ and uses again $\d(\V{Z}_0) \le k$,
the fourth inequality uses $i \le \codim(\V{Z}_i)$ $(\le A+1)$ which follows from the first item of~(\ref{eq:Z-prop}), 
and the last inequality uses~(\ref{eq:F-size}).
%
%
%
%
Thus, again by~(\ref{eq:AR-F-bd}),
$\AR(T) \ge \codim\V{Z}_{i+1} - 1/C$, 
completing the induction step.
We obtain, for $i=t$, that
\begin{equation}\label{eq:AR-bound}
\AR(T) \ge \codim\V{Z} - 1/C . 
\end{equation}
Finally, combining~(\ref{eq:PR-bound}) and~(\ref{eq:AR-bound}), we deduce that
$$\PR(T) \le (2^{k-1}-1)(\AR(T)+1/C)$$
provided~(\ref{eq:F-size}) holds.
Taking $C=2^{k-1}-1$ (so $|\FF| \ge 2^{2^{O(\AR(T)+k)}}$) completes the proof.
\end{proof}

Next we prove a that quantitative improvement of Theorem~\ref{theo:main} already yields an almost linear bound for $\PR$ in terms of $\AR$ that holds over every finite field.

\begin{proof}[Proof of Corollary~\ref{coro:main}]
	Let $\FF$ be any finite field, let
	$\KK$ be the extension field of $\FF$ of degree $\ell$ (so $|\KK|=|\FF|^\ell$),  
	and let $T$ be a $k$-tensor over $\FF$.
	We next show that $\PR(T) \le \ell\PR_\KK(T)$, 
	and that $\AR_\KK(T) \le \ell^k\AR(T)$.
	This would complete the proof since applying Theorem~\ref{theo:main} on $T$, viewed as a tensor over the field $\KK$ of degree  $\ell=\lceil{\log_{|\FF|}F(\AR_\KK(T),k)}\rceil$, yields that
	$\PR(T) \le O_k(\ell^{k+1}(\AR(T)+1))$.
	Therefore, if Theorem~\ref{theo:main} holds with $F(r,k) \le \exp(O_k(r^{o(1)}))$ then we have
	$\PR(T) \le O_k(\AR(T)^{1+o(1)})$ over every finite field.
	
	Henceforth write $\KK=\FF[\alpha]$, 
	and let $\Tr \colon \KK\to\FF$ be the $\FF$-linear function satisfying $\Tr(x)=x$ for every $x \in \FF$
	(i.e., $\Tr(1)=1$)
	given by  
	$\Tr(x) = \Tr_{\KK/\FF}(x_0x)$, 
	where $\Tr_{\KK/\FF}(x)=\sum_{i=0}^{\ell-1} x^{|\FF|^i} \colon \KK \to \FF$ is the trace of $\KK$ on $\FF$ and
	$x_0 \in \KK$ satisfies $\Tr_{\KK/\FF}(x_0)=1$ (as $\Tr_{\KK/\FF}$ is surjective).
	We naturally extend $\Tr$ to $\KK$-polynomials, $\Tr\colon\KK[\B{x}]\to\FF[\B{x}]$, by acting on the coefficients.
	We first prove the inequality $\PR(T) \le \ell\PR_\KK(T)$.
	Note that if $p,q$ are $\KK$-polynomials, we can write 
	$p=\sum_{t=0}^{\ell-1} \alpha^t p^{(t)}$ 
	and $q=\sum_{t=0}^{\ell-1} \alpha^t q^{(t)}$,
	where for each $t$ we have that $p^{(t)}$ and $q^{(t)}$ are $\FF$-polynomials whose support is contained in the support of $p$ and $q$, respectively. 
	Then $\Tr(pq) = 
	\Tr(\sum_{t,s=0}^{\ell-1} \alpha^{t+s}p^{(t)}q^{(s)}) 
	=\sum_{t=0}^{\ell-1} p^{(t)}q'^{(t)}$ with $q'^{(t)}=\sum_{s=0}^{\ell-1} \Tr(\alpha^{t+s}) q^{(s)}$.
	Since $T=\Tr(T)$, and by the linearity of $\Tr$, the desired inequality $\PR(T) \le \ell\PR_\KK(T)$ follows.
	
	
	It remains to prove the inequality $\AR_\KK(T) \le \ell^k\AR(T)$.
	Note that if $\chi$ is a nontrivial additive character of $\FF$ then $\chi\circ\Tr$ is a nontrivial additive character of $\KK$.
	Therefore, by the definition of analytic rank, $\AR_\KK(T) \le \AR(\Tr(T'))$ where $T \colon (\FF^n)^k \to \FF$, and $T' \colon (\FF^{\ell n})^{k}\to\KK$ is given by
	$$T'((\B{x}^{(t)}_1)_{t=0}^{\ell-1},\ldots,(\B{x}^{(t)}_k)_{t=0}^{\ell-1}) 
	= T\Big(\sum_{t=0}^{\ell-1} \alpha^t \B{x}^{(t)}_1,\ldots,\sum_{t=0}^{\ell-1} \alpha^{t} \B{x}^{(t)}_k\Big).$$
	%
	Considering each monomial separately, we have 
	$$\AR_\KK(T) \le \AR(\Tr(T')) 
	= \AR\Big(\sum_{t_1,\ldots,t_k=0}^{\ell-1} \Tr(\alpha^{t_1+\cdots+t_k})T(\B{x}^{(t_1)}_1,\ldots,\B{x}^{(t_k)}_k)\Big)
	\le \ell^k \AR(T)$$ 
	using the sub-additivity of analytic rank~\cite{Lovett19}.
	This completes the proof.
\end{proof}

\subsection{Proof of a polynomial bound for the $d$~vs.~$d-1$ Gowers Inverse Conjecture}\label{subsec:Gowers-proof}

We finish by deducing Corollary~\ref{coro:Gowers}.
For the rest of this subsection $\FF$ is a finite field and
$\chi \colon \FF\to\Comp$ is a nontrivial additive character 
(so $\chi(y) = (e^{2\pi i/\ch(\FF)})^{\Tr(cy)}$ with $0 \neq c \in \FF$, where $\Tr$ here is the trace of $\FF$ on $\FF_{\ch(\FF)}$).
For any function $p \colon \FF^n \to \FF$, denote by $\Delta_{\B{v}}p(\B{x}) = p(\B{x}+\B{v})-p(\B{x})$ the \emph{additive derivative} of $p$ along $\V{v} \in \FF^n$.
%
For a polynomial $p\colon\FF^n\to\FF$ of degree $d$, consider the symmetric $d$-tensor $\widetilde{p} \colon (\FF^n)^d \to \FF$ given by (see~\cite{GowersWo11}, Lemma~2.4)
$$\widetilde{p}(\B{v}_1,\ldots,\B{v}_d) 
:= \Delta_{\B{v}_1}\cdots\Delta_{\B{v}_d} p
= \sum_{S \sub [d]} (-1)^{d-|S|}p\big(\sum_{i \in S} \B{v}_i\big) .$$
%
%
For a polynomial $p \colon \FF^n\to\FF$, its \emph{rank} $\rank(p)$ is defined as the smallest number $r$ of polynomials $q_1,\ldots,q_r \colon \FF^n\to\FF$ with $\deg(q_i)<\deg(p)$ such that $p=\phi(q_1,\ldots,q_r)$ for some function $\phi\colon\FF^r\to\FF$.
We have the following standard properties (see, e.g.,~\cite{GreenTao09,Janzer19}), where $p\colon\FF^n\to\FF$ is of degree $d$ and $f(\B{x})=\chi(p(\B{x}))$:
%
\begin{enumerate}
	\item\label{item:Gowers1} $\norm{f}_{U^d}^{2^d} = |\FF|^{-\AR(\widetilde{p})}$,
	\item\label{item:Gowers2} If $d<\ch(\FF)$ then 
	$\rank(p) \le 2\PR(\widetilde{p})+1$.
\end{enumerate}
The first item follows from $\Delta^*_{\B{v}_1}\cdots\Delta^*_{\B{v}_d} \chi(p(\B{x})) = \chi(\Delta_{\B{v}_1}\cdots\Delta_{\B{v}_d} p(\B{x}))
= \chi(\widetilde{p}(\B{v}_1,\ldots,\B{v}_d))$.
The second item follows from the Taylor expansion 
$p(\B{x}) = (d!)^{-1}\widetilde{p}(\B{x},\ldots,\B{x}) + q(\B{x})$
with $\deg q < d$, 
as $(d!)^{-1} \in \FF$,
which implies $\rank(p) \le \rank(\widetilde{p})+1 \le 2\PR(\widetilde{p})+1$.
We can now easily obtain, from Theorem~\ref{theo:main}, the polynomial bound  
$\norm{f}_{u^d} \ge \norm{f}_{U^d}^{8^d}$ 
(recall~(\ref{eq:Gowers-norms})).
\begin{proof}[Proof of Corollary~\ref{coro:Gowers}]
	We assume $d \ge 2$, since for $d=1$ the statement is trivially true as $\norm{f}_{U^1} = |\Ex_{\B{x}\in\FF^n} f(\B{x})| \le \norm{f}_{u^1}$.
	It suffices to prove that
	\begin{equation}\label{eq:poly-ranks-bd}
	\rank(p) \le 2^{d+1}\AR(\widetilde{p}).
	\end{equation}
	Indeed,~(\ref{eq:poly-ranks-bd}) would imply
	$$\norm{f}_{u^d} \ge |\FF|^{-\rank(p)} 
	\ge |\FF|^{-2^{d+1}\AR(\widetilde{p})}
	= \norm{f}_{U^d}^{2^d2^{d+1}}
	\ge \norm{f}_{U^d}^{8^d}$$
	where the first inequality follows by averaging using the Fourier expansion of $f$,\footnote{
		Observe that every additive character $\psi \colon \FF^r \to \mathbb{C}$ is of the form $\psi(\B{y})=\chi(\boldsymbol\alpha \cdot \B{y})$ for some $\boldsymbol\alpha \in \FF^r$; 
		indeed, $\psi(\B{y}) = \omega^{\Tr(\boldsymbol\gamma \cdot \B{y})}$ for some $\boldsymbol\gamma \in \FF^r$, where $\omega = e^{2\pi i/\ch(\FF)}$, 
		so if $\chi(y) = \omega^{\Tr(cy)}$ $(c \neq 0)$ then
		$\psi(\B{y}) = \chi(\boldsymbol\alpha\cdot\B{y})$ with $\boldsymbol\alpha = c^{-1}\boldsymbol\gamma$. 
		The averaging argument is roughly as follows (see Theorem~1.4 in~\cite{GreenTao09} or Theorem~1.6 in~\cite{Janzer19}): write $f=\phi(q_1,\ldots,q_r)$ and expand 
		$\phi(\B{y}) = \sum_\psi \hat{\phi}(\psi)\psi(\B{y})$ where the sum is over the characters of $\FF^r$; 
		then $f = \sum_\psi \hat{\phi}(\psi)\psi(q_1,\ldots,q_r)$, 
		and so $1 = \Ex_{\B{x} \in \FF^n} |f(\B{x})|^2 
		= \sum_\psi \overline{\hat{\phi}(\psi)}\Ex_{\B{x} \in \FF^n} f(\B{x})\overline{\psi(q_1(\B{x}),\ldots,q_r(\B{x}))}$ by rearranging;
		by averaging, $|\Ex_{\B{x} \in \FF^n} f(\B{x})\overline{\psi(q_1(\B{x}),\ldots,q_r(\B{x}))}| \ge |\FF|^{-r}$
		for some character $\psi$, and we are done by the observation above.
	}
	and the equality is Property~(\ref{item:Gowers1}).
%
	To prove~(\ref{eq:poly-ranks-bd}), we have
	$$\rank(p) \le 2\PR(\widetilde{p})+1 
	\le (2^d-2)\AR(\widetilde{p})+3$$
	where the first inequality is Property~(\ref{item:Gowers2}) and the second inequality applies Theorem~\ref{theo:main}, 
	which means we assume $\ch(\FF)>d$ and $|\FF| \ge F(\AR(\widetilde{p}),d)
	=: F'(1/\norm{f}_{U^d},d)$.
	However, we must remove the additive term from the right hand side 
	to obtain the bound on $\norm{f}_{u^d}$ in the statement, which does not involve $|\FF|$. 
	We show that this can be done since $|\FF|$ is large enough.
	By the Schwartz-Zippel lemma, any nonzero $\FF$-polynomial $q\colon\FF^m\to\FF$ of degree $k$ has at most $(k/|\FF|)\cdot|\FF|^m$ zeros.
	It follow that any nonzero $d$-tensor $T \colon (\FF^n)^d \to \FF$ ($d \ge 2$) satisfies $|\ker\sl{T}(\FF)| \le ((d-1)/|\FF|)\cdot|\FF|^{(d-1)n}$; 
	indeed, this holds for each nonzero polynomial of the slicing polynomial map $\sl{T} \colon (\FF^n)^{d-1}\to\FF^n$, which is a polynomial of degree $d-1$.
	We deduce that 
	$$\AR(\widetilde{p}) = \log_{|\FF|}(|\FF|^{(d-1)n}/|\ker\sl{\widetilde{p}}(\FF)|) \ge 1 - \log_{|\FF|}(d-1) 
	\ge 1/2$$ 
	where the last inequality uses the assumption that $|\FF|$ is large enough as a function of $d$.
	This gives~(\ref{eq:poly-ranks-bd})
	via the bound 
	$(2^d-2)\AR(\widetilde{p})+3 \le 2^{d+1}\AR(\widetilde{p})$
	(equivalently, $3 \le (2^d+2)\AR(\widetilde{p})$),
	completing the proof.
	%
	%
	%
	%
\end{proof}

We finally note that an easy corollary of Theorem~\ref{theo:main} 
is the inverse theorem for analytic rank vs.\ rank for general polynomials over large fields: for $p\colon\FF^n\to\FF$ a polynomial of degree $d$, its rank $\rank(p)$ and analytic rank $\AR_\chi(p):=-\log_{|\FF|}|\Exp_{\B{x}\in\FF^n} \chi(p(\B{x}))|$ satisfy
$$\rank(p) \le 4^{d+1}\AR_\chi(p)$$
if $\ch(\FF)>d$ and $|\FF|$ is large enough as a function of $\AR(p)$ and $d$.
This is immediate from~(\ref{eq:poly-ranks-bd}) and $\AR(\widetilde{p}) \le 2^d\AR_\chi(p)$; the latter inequality follows from Property~(\ref{item:Gowers1}) with $f(\B{x})=\chi(p(\B{x}))$ together with monotonicity $\norm{f}_{U^d} \ge \norm{f}_{U^1}$ (e.g.,~\cite{GreenTao08} Equation~(1.2)) and the identity $\norm{f}_{U^1}^2 = |\Ex_{\B{v},\B{x}\in\FF^n} \Delta^*_{\B{v}} f(\B{x})| = |\Ex_{\B{x}\in\FF^n} f(\B{x})|^2 = (|\FF|^{-\AR_\chi(p)})^2$ (see~\cite{GreenTao09,Janzer19}).
For large fields $\FF$ as above, 
this gives a proof of, e.g., Theorem~1.8 of Green and Tao~\cite{GreenTao09}, or Conjecture 1.16 of Tao and Ziegler~\cite{TaoZi12}, with essentially an optimal bound.

\subsection*{Acknowledgments}
We thank Zeyu Guo, Aise Johan de Jong, David Kazhdan, Will Sawin, Avi Wigderson, Daniel Zhu, Tamar Ziegler, and Jeroen Zuiddam for helpful discussions.
We thank the anonymous referees for their careful reading and useful comments.

\end{document}